\documentclass[review,3p,square,sorted]{elsarticle}

\usepackage{amssymb}
\usepackage{amsmath}
\usepackage{amsthm}
\usepackage{stmaryrd}
\usepackage{bbm}
\usepackage{enumitem}

\usepackage{array}
\usepackage{multirow}
\usepackage{makecell}
\usepackage{longtable}
\usepackage{caption}

\usepackage{fancyhdr}
\usepackage{hyperref}
\hypersetup{%
  pdftitle={BSDEs},%
  pdfsubject={BSDEs},%
  pdfauthor={ShengJun FAN, Ying Hu},%
  pdfkeywords={BSDEs},%
  pdfstartview=FitH,%
  CJKbookmarks=true,%
  bookmarksnumbered=true,%
  bookmarksopen=true,%
  colorlinks=true, linkcolor=blue, urlcolor=blue, citecolor=blue, %
}

\usepackage{cleveref}
\crefname{thm}{Theorem}{Theorems}
\crefname{pro}{Proposition}{Propositions}
\crefname{lem}{Lemma}{Lemmas}
\crefname{rmk}{Remark}{Remarks}
\crefname{cor}{Corollary}{Corollaries}
\crefname{dfn}{Definition}{Definitions}
\crefname{ex}{Example}{Examples}
\crefname{section}{Section}{Sections}
\crefname{subsection}{Subsection}{Subsections}

\newcommand{\eps}{\varepsilon}

\newcommand{\To}{\rightarrow}
\newcommand{\as}{{\rm d}\mathbb{P}\times{\rm d} t-a.e.}

\newcommand{\ps}{\mathbb{P}-a.s.}
\newcommand{\dif}{{\rm d}}

\newcommand{\tim}{\times}
\newcommand{\essinf}{\mathop{\operatorname{ess\,inf}}}

\newcommand{\F}{\mathcal{F}}
\newcommand{\E}{\mathbb{E}}

\newcommand{\acal}{\mathcal{A}}

\newcommand{\qcal}{\mathcal{Q}}

\newcommand{\T}{[0,T]}

\newcommand{\p}{{\mathbb P}}
\newcommand{\Q}{{\mathbb Q}}
\newcommand{\R}{{\mathbb R}}
\newcommand{\RE}{\forall}

\newcommand {\Dis}{\displaystyle}

\newtheorem{thm}{Theorem}[section]

\newtheorem{pro}[thm]{Proposition}
\newtheorem{rmk}[thm]{Remark}

\newtheorem{ex}[thm]{Example}

\journal{arXiv}
\begin{document}
\begin{frontmatter}

\title{{Unbounded Dynamic Concave Utilities via BSDEs}\tnoteref{found}}
\tnotetext[found]{This work is supported by National Natural Science Foundation of China (Nos. 12631016, 12171471, 12031009 and 11631004), by Key Laboratory of Mathematics for Nonlinear Sciences (Fudan University), Ministry of Education, Handan Road 220, Shanghai 200433, China, and by Lebesgue Center of Mathematics ``Investissements d'avenir" program-ANR-11-LABX-0020-01.
\vspace{0.2cm}}


\author[Fan]{Shengjun Fan} \ead{shengjunfan@cumt.edu.cn}
\author[Hu]{Ying Hu} \ead{ying.hu@univ-rennes1.fr}
\author[Tang]{Shanjian Tang} \ead{sjtang@fudan.edu.cn} \vspace{-0.5cm}

\affiliation[Fan]{organization={School of Mathematics, China University of Mining and Technology},
            city={Xuzhou 221116},
            country={China}}

\affiliation[Hu]{organization={Univ. Rennes, CNRS, IRMAR-UMR6625},
            city={F-35000, Rennes},
            country={France}}

\affiliation[Tang]{organization={Department of Finance and Control Sciences, School of Mathematical Sciences, FudanUniversity},
            city={Shanghai 200433},
            country={China}}


\begin{abstract}
The dynamic concave utility (or the dynamic convex risk measure) of an unbounded  endowment is studied and represented as the value process in the unique solution of a backward stochastic differential equation (BSDE) with an unbounded terminal value, with the help of our recent existence and uniqueness results on unbounded solutions of scalar BSDEs whose generators have a linear, super-linear, sub-quadratic or quadratic growth. Moreover, the infimum in the dynamic concave utility is shown to be achieved in a well-chosen admissible set associated with the terminal value and the core function. The Fenchel-Legendre conjugates of convex functions, the de la Vall\'{e}e-Poussin theorem, and Young's and Gronwall's inequalities constitute the main ingredients of the dual representation.\vspace{0.2cm}
\end{abstract}

\begin{keyword}
Dynamic concave utilities\sep  Dynamic convex risk measure\sep Dual representation\sep \\
\hspace*{1.8cm} Backward stochastic differential equation\sep Unbounded terminal value.
\vspace{0.2cm}

\MSC[2010] 60H10, 60G44, 91B30\vspace{0.2cm}
\end{keyword}

\end{frontmatter}
\vspace{-0.4cm}

\section{Introduction}
\label{sec:1-Introduction}
\setcounter{equation}{0}

Utility or risk measure is defined via axioms to characterize the preference of an economic agent, or the risk of a random endowment (a contingent claim), via phenomenological properties of the preference in economics or the risk  in risk management. The reader is referred to \cite{DuffieEpstein1992Econometrica,ElKarouiPengQuenez1997MF, HuImkeller2005AAP,DelbaenPengRosazza2010FS,DelbaenHuBao2011PTRF,Delbaen2012Osaka,
BeissnerLinRiedel2020MF} for utilities, \cite{Artzner1999MF,Delbaen2000LN,Delbaen2002AFS} for coherent risk measures, \cite{FrittelliRosazza2002JBF,FollmerSchied2002FS,FollmerSchied2002AFS,
Cheridito-Delbaen-Kupper2004SPA,Cheridito-Delbaen-Kupper2005FS,SmithBickel2022OR} for convex risk measures, \cite{MaoWang2020SIAMFM,JiaXiaZhao2020ArXiv} for monetary risk measures, and \cite{Castagnoli2022OR,Laeven2023arXiv,TianWang2023arXiv} for star-shaped risk measures. The dynamic version of this theme, that is the dynamic utility or the dynamic risk measure, as a time-parameterized family of  operators defined on spaces of random variables, is a popular notion in financial mathematics, and has received an extensive attention for example in \cite{Peng2004LectureNotes,FrittelliRosazza2004book,DetlefsenScandolo2005FS,
Rosazza2006IME,FollmerPenner2006SD,KloppelSchweizer2007MF,Bion-Nadal2008FS,
Jiang2008AAP,Bion-Nadal2009SPA,Drapeau2016AIHPPS,
Laeven2023arXiv,TianWang2023arXiv,RosazzaZullino2024PUQR}. The paper focuses on a dynamic concave utility, or equivalently a dynamic convex risk measure, with an emphasis on the unboundedness of  elements (as random variables) of the family of operators' domains.\vspace{0.2cm}

Fix a nonnegative real number $T>0$ and an integer $d\geq 1$. Assume that $(B_t)_{t\in\T}$ is a standard $d$-dimensional Brownian motion defined on some
complete probability space $(\Omega, \F, \mathbb{P})$, $(\F_t)_{t\geq 0}$ is its natural and augmented filtration, and $\F_T=\F$. An equality or inequality between random elements is understood to hold true in the $\ps$ sense. For each $t\in\T$, $L^\infty(\F_t)$ is the set of $\F_t$-measurable scalar bounded random variables, and $\acal(\F_t)$ is a general linear space of $\F_t$-measurable scalar random variables containing $L^\infty(\F_t)$. Suppose that $\acal(\F_s)\subset \acal(\F_t)$ for each $0\leq s\leq t\leq T$. By a dynamic concave utility on $\acal(\F_T)$, we mean a family of time-parameterized operators $\{U_t(\cdot), t\in\T\}$ { mapping} from $\acal(\F_T)$ to $\acal(\F_t)$ such that the following properties are satisfied for each $t\in\T$:\vspace{0.2cm}

(i) {\bf Preservation of constants}: $U_t(c)=c$ for each constant $c$;

(ii) {\bf Monotonicity}: $U_t(\xi)\geq U_t(\eta)$ for each $\xi,\eta\in \acal(\F_T)$ such that $\xi\geq \eta$;

(iii) {\bf Translation invariance}: $U_t(\xi+\eta)=U_t(\xi)+\eta$ for each $\xi\in \acal(\F_T)$ and $\eta\in \acal(\F_t)$;

(iv) {\bf Concavity}: $U_t(\theta\xi+(1-\theta)\eta)\geq \theta U_t(\xi)+(1-\theta)U_t(\eta)$ for all $\xi,\eta\in \acal(\F_T)$ and $\theta\in (0,1)$;

(v) {\bf Time consistency}: $U_s(\xi)=U_s(U_t(\xi))$ for each $\xi\in \acal(\F_T)$ and $s\in [0,t]$.\vspace{0.2cm}

\noindent It is noteworthy that if $\{U_t(\cdot), t\in\T\}$ is a dynamic concave utility on $\acal(\F_T)$ and $\rho_t(X):=-U_t(-X)$ for each $t\in\T$ and $\xi\in \acal(\F_T)$, then $\{\rho_t(\cdot), t\in\T\}$ defines a dynamic convex risk measure on $\acal(\F_T)$. As a result, all assertions obtained in this paper on dynamic concave utilities have counterparts of dynamic convex risk measures.\vspace{0.2cm}

It is well known that the theory of Backward Stochastic Differential Equation (BSDEs) established by \citet{PardouxPeng1990SCL} and \citet{Kobylanski2000AP} is a fundamental tool to study time consistent dynamic utilities, risk measures and nonlinear expectations. See for example \cite{ElKarouiPengQuenez1997MF,CoquetHuMeminPeng2002PTRF, Peng2004LectureNotes,HuImkeller2005AAP,Tang2006CRA,Rosazza2006IME,
Jiang2008AAP,HuMaPengYao2008SPA,BarrieuElKaroui2009Book,DelbaenPengRosazza2010FS,
DelbaenHuBao2011PTRF,Drapeau2016AIHPPS,Faidi2017IJYAF,
JiShiWangZhou2019IME,TianWang2023arXiv,FanHuTang2026MF,
Laeven2023arXiv,RosazzaZullino2024PUQR} for details. To understand the essential nature of a dynamic concave utility, we are particularly interested in its mathematical representation as an adapted solution of a BSDE. Most of the literature in this issue focused on bounded endowments. In what follows, let us recall some of them.\vspace{0.2cm}

Given a lower semi-continuous (LSC) convex function $f:\R^d\To \R_+\cup \{+\infty\}$ such that $f(0)=0$ and $\liminf_{|q|\To +\infty} f(q)/|q|^2>0$, we define for bounded endowment $\xi\in L^\infty(\F_T)$,\vspace{0.1cm}
\begin{equation}\label{eq:1.1}
U_t(\xi):=\essinf\limits_{q\in \qcal_1}\ \E_{\Q^q}\left[\left. \xi+\int_t^T f(q_s){\rm d}s\right|\F_t\right],\quad t\in \T,\vspace{0.1cm}
\end{equation}
where
\begin{equation}\label{eq:1.2}
\begin{array}{ll}
\qcal_1:=\bigg\{&\Dis (q_t)_{t\in[0,T]}~\text{is an } (\F_t)\text{-progressively measurable~} \R^d \text{-valued process such that} \\
&\Dis \int_{0}^{T}|q_s|^2\dif s<+\infty, ~~\ps; \ \ L_t^q:=\exp\Big(\int_{0}^{t}q_s\cdot \dif B_s-\frac{1}{2}\int_{0}^{t}|q_s|^2\dif s\Big),\ 0\le t\le T, \vspace{0.2cm}\\
&\Dis \text{is a martingale; and} \ \frac{\dif \Q^q}{\dif \mathbb{P}}:=L_T^q \in L^1(\F_T)\ \ \bigg\},\vspace{0.2cm}
\end{array}
\end{equation}
$x\cdot y$ denotes the Euclidean inner product of two vectors $x,y\in \R^d$, and $\E_{\Q^q}[\cdot|\F_t]$ is the expectation operator conditioned on the $\sigma$-field $\F_t$ under the probability equivalent measure $\Q^q$.
The operator-valued process $\{U_t(\cdot), ~t\in \T\}$ defined via \eqref{eq:1.1} is easily verified to be a dynamic concave utility on $L^\infty(\F_T)$. Inversely, as observed in \citet{DelbaenPengRosazza2010FS}, any dynamic concave utility on $L^\infty(\F_T)$ has a representation like  \eqref{eq:1.1} under some extra mild assumptions. Furthermore, Theorems 2.1-2.2 in \citet{DelbaenHuBao2011PTRF} show that there is an $(\F_t)$-progressively measurable square-integrable $\R^d$-valued process $(Z_t)_{t\in\T}$ such that $(U_t(\xi), Z_t)_{t\in\T}$ is the unique bounded solution of the following scalar BSDE:
\begin{equation}\label{eq:1.3}
Y_t=\xi-\int_t^T g(Z_s){\rm d}s+\int_t^T Z_s \cdot {\rm d}B_s,\quad t\in \T,
\end{equation}
where the function
\begin{equation}\label{eq:1.4}
g(z):=\sup_{q\in \R^d}(z\cdot q-f(q))\geq 0,\quad  z\in \R^d\vspace{0.1cm}
\end{equation}
is conjugate to $f$ and {  convex}, and has the properties that  $g(0)=0$ and $\limsup_{|z|\To +\infty} g(z)/|z|^2<+\infty$. In other words, for a nonnegative, LSC, convex and super-quadratically growing function $f$ with $f(0)=0$, the dynamic concave utility on $L^\infty(\F_T)$ defined via \eqref{eq:1.1} has a dual representation as the value process in the unique bounded solution to BSDE \eqref{eq:1.3}. The reader is referred to \citet{Faidi2017IJYAF} for the case of an unbounded endowment $\xi$ of  finite exponential moment of arbitrary order, but where no attainability is addressed.\vspace{0.2cm}

Dual representation of an unbounded dynamic convex risk measure via BSDE was investigated in \citet{Drapeau2016AIHPPS}. More specifically, in terms of dynamic concave utilities, Theorem 4.5 in \cite{Drapeau2016AIHPPS} {\color{red}gives} the following assertion. Given an LSC convex function $f:\R^d\To \R_+\cup \{+\infty\}$ such that $f(0)=0$ and
$$\acal_0(\F_T):=\{\xi\in\F_T|\sup_{t\in\T}|\E[\xi|\F_t]|\in L^1\}{\color{red},}$$
we define for $\xi\in \acal_0(\F_T)$,
\begin{equation}\label{eq:1.5}
    U_t(\xi):=\essinf\limits_{q\in \qcal_2}\E_{\Q^q}\left[\left. \xi+\int_t^T f(q_s){\rm d}s\right|\F_t\right],\ \ t\in \T,\vspace{0.1cm}
\end{equation}
with
\begin{equation}\label{eq:1.6}
\begin{array}{ll}
\qcal_2:=\bigg\{&\Dis (q_t)_{t\in[0,T]}~\text{is an } (\F_t)\text{-progressively measurable~} \R^d \text{-valued process such that }\\
&\Dis \int_{0}^{T}|q_s|^2\dif s<+\infty, ~~\ps;\ \ L_t^q:=\exp\Big(\int_{0}^{t}q_s\cdot \dif B_s-\frac{1}{2}\int_{0}^{t}|q_s|^2\dif s\Big), \ 0\le t\le T, \vspace{0.2cm}\\
&\Dis \text{is a martingale, and} \ \frac{\dif \Q^q}{\dif \mathbb{P}}:=L_T^q \in L^\infty(\F_T)\ \ \bigg\}.\vspace{0.2cm}
\end{array}
\end{equation}
Then, $Y_t:=U_t(\xi)$ is well-defined and there is an $(\F_t)$-progressively measurable $\R^d$-valued process $(Z_t)_{t\in\T}$ such that $(Y_t, Z_t)_{t\in\T}$ is the unique minimal supersolution of a convex BSDE with the terminal value $-\xi$ and the driver $g$ defined in \eqref{eq:1.4} (see \citet{Drapeau2013AP}): that is, $(Y_t)_{t\in\T}$ is $(\F_t)$-adapted and c\`{a}dl\`{a}g, $(Z_t)_{t\in\T}$ is $(\F_t)$-progressively measurable such that $\int_0^T |Z_t|^2{\rm d}t<+\infty$ and $(\int_0^t Z_s\cdot {\rm d}B_s)_{t\in\T}$ is a supermartingale, and $(Y_t, Z_t)_{t\in\T}$ fulfils
\begin{equation}\label{eq:1.7}
\Dis Y_s\geq Y_t+\int_s^t g(Z_u){\rm d}u-\int_s^t Z_u\cdot {\rm d}B_u, \ \ \RE\ 0\leq s\leq t\leq T,\ \ {\rm and}\ \ Y_T\geq -\xi;
\end{equation}
moreover, if $(\tilde Y_t, \tilde Z_t)_{t\in\T}$ also satisfies the above conditions, then $Y_t\leq \tilde Y_t$ for each $t\in\T$. Note that a  minimal supersolution $(Y_t, Z_t)_{t\in\T}$ of BSDE \eqref{eq:1.7} might not be  an  adapted solution $(\bar Y_t, \bar Z_t)_{t\in\T}$ of the corresponding BSDE
\begin{equation}\label{eq:1.8}
\bar Y_t=-\xi+\int_t^T g(\bar Z_s){\rm d}s-\int_t^T \bar Z_s\cdot {\rm d}B_s, \ \ t\in\T.
\end{equation}
However, Theorem 4.6 of \citet{Drapeau2016AIHPPS} asserts that  if the infimum in $U_0(\xi)$ of \eqref{eq:1.5} is achieved at some $q\in \qcal_2$, the unique minimal supersolution $(-U_t(\xi), Z_t)_{t\in\T}$ of BSDE \eqref{eq:1.7} is actually an adapted solution of BSDE \eqref{eq:1.8}  and equivalently  BSDE \eqref{eq:1.3}, and for each $t\in\T$, $q_t\in \partial g(Z_t)$ and $U_t(\xi)$ is also achieved at $q\in \qcal_2$. In other words, for a nonnegative, LSC and convex function $f$ with $f(0)=0$, the dynamic concave utility on $\acal_0(\F_T)$ defined via \eqref{eq:1.5} has a dual representation as the value process in the unique adapted solution to BSDE \eqref{eq:1.3} provided that the infimum of $U_0(\xi)$ in \eqref{eq:1.5} is attained within $\qcal_2$.
However, the infimum in $U_0(\xi)$ of \eqref{eq:1.5} is not necessarily achieved within $\qcal_2$, and it is generally hard to verify the attainability of the infimum since the stochastic  exponential process $L_t^q$ of $ q\in \qcal_2$ is required to be essentially bounded.\vspace{0.2cm}

On the other hand, consider a nonnegative, LSC, convex and super-quadratically growing function $f$ with $f(0)=0$. For any unbounded endowment $\xi$ of a finite exponential moment of arbitrary order, \citet{BriandHu2006PTRF,BriandHu2008PTRF} confirm that BSDE \eqref{eq:1.3} still has a unique adapted solution $(Y_t,Z_t)_{t\in \T}$ such that $\sup_{t\in \T} |Y_t|$ has a finite exponential moment of arbitrary order, and the associated comparison theorem and stability theorem on the solutions of preceding BSDEs hold true.
Then, the following questions are naturally asked: for an endowment $\xi$ admitting a finite exponential moment of arbitrary order, is there a set $\qcal(\xi,f)$ of density processes $(q_t)_{t\in\T}$ of probability equivalent measures $\Q^q$ (like $\qcal_1$ and $\qcal_2$ in \eqref{eq:1.2} and \eqref{eq:1.6}) associated with $(\xi,f)$ such that the following operators $U_t$ with
\begin{equation}\label{eq:1.9}
U_t(\xi):=\essinf\limits_{q\in \qcal(\xi,f)}\E_{\Q^q}\left[\left. \xi+\int_t^T f(s,q_s){\rm d}s\right|\F_t\right],\quad  t\in \T,
\end{equation}
are well-defined and represented as the value process in the unique adapted solution of BSDE \eqref{eq:1.3}, and moreover,  the infimum in \eqref{eq:1.9} is attainable at some $q\in \qcal(\xi,f)$? If the answer is affirmative for some $\qcal(\xi,f)$, then the operator process  $\{U_t(\cdot),\ t\in\T\}$ is a dynamic concave utility on a space of possibly unbounded endowments admitting a finite exponential moment of arbitrary order. Intuitively, the set $\qcal(\xi,f)$ is neither too large like $\qcal_1$ nor too small like $\qcal_2$ in order to ensure both well-posedness and attainability of \eqref{eq:1.9}. Furthermore, alternative question arises  naturally: are there larger linear spaces of possibly unbounded endowments $\xi$ connected to different features of functions $f$ and sets $\qcal(\xi,f)$ where the preceding duality representation and attainability of the infimum remain true? In this paper, we shall give some affirmative answers to both questions.\vspace{0.2cm}

More specifically, we first present recent existence and uniqueness results on unbounded solutions of scalar BSDEs whose generators have a linear, super-linear, sub-quadratic or quadratic growth established in \cite{BriandHu2008PTRF,HuTang2018ECP,BuckdahnHuTang2018ECP,FanHu2019ECP, FanHu2021SPA,OKimPak2021CRM,FanHuTang2023SPA,FanHuTang2026MF}, where the possibly unbounded terminal value $\xi$ belongs to various linear space $\acal(\F_T)$ containing $L^\infty(\F_T)$. See \cref{pro:2.2} in Section 2 for details. Then, with the help of these results, under four different scenarios of the endowment and core function $(\xi,f)$ that are respectively linked to the linear, super-linear, sub-quadratic or quadratic growth on the generator $g$ of BSDE~\eqref{eq:1.3}, we prove that for a suitable set $\qcal(\xi,f)$, the dynamic concave utility of a possibly unbounded endowment $\xi$ is well-defined by \eqref{eq:1.9} and it is represented as the value process in the unique adapted solution of BSDE~\eqref{eq:1.3}, and moreover, the infimum in \eqref{eq:1.9} is attainable. In particular, we consider the general case of core function $f$, which might be time-varying and random. See \cref{thm:3.1} in Section 3 for details. \cref{thm:3.1} extends the corresponding result of \citet{DelbaenPengRosazza2010FS} and \citet{DelbaenHuBao2011PTRF} to the case of an unbounded endowment. Some results obtained in \citet{Drapeau2016AIHPPS} are also improved. More crucially, \cref{thm:3.1} enables us to compute the unbounded dynamic concave utilities via numerical solution of BSDEs such as the Monte Carlo method. See \cref{rmk:3.4} in Section 3 for more details. In addition, in \cref{ex:3.3} of Section 3 we present several examples to illustrate \cref{thm:3.1}.\vspace{0.2cm}

The proof of \cref{thm:3.1} is quite involved. It is distinguished from those used in \cite{Bion-Nadal2008FS,DetlefsenScandolo2005FS, DelbaenPengRosazza2010FS,DelbaenHuBao2011PTRF,Drapeau2016AIHPPS} where truncation and approximation arguments, the BMO-martingale theory, Doob-Meyer decomposition of a supermartingale and some supermartingale property are conventionally  employed. For our case, with the help of Young's and Gronwall's inequalities, we first propose and prove two useful technical propositions (see \cref{pro:4.1,pro:4.3} in Section 4), where new test functions have to be calculated with It\^o's formula. Then, by virtue of both propositions along with \cref{pro:2.2}, using the standard duality arguments such as the Fenchel-Legendre transform of a convex functional and the Fenchel-Moreau theorem as well as the de la Vall\'{e}e-Poussin theorem,  we verify the desired dual representation in \cref{thm:3.1}. We would like to emphasize that the construction of $\qcal(\xi,f)$, \cref{pro:4.1,pro:4.3}, and the attainability argument in Section 4 are the main contributions of this \vspace{0.2cm}paper.

The rest of the paper is organized as follows. In Section 2, we introduce notations, spaces and preliminary existence and uniqueness results { (\cref{pro:2.2})} on adapted solutions to scalar BSDEs whose generators have a linear, super-linear, sub-quadratic or quadratic growth. In Section 3, we state \cref{thm:3.1} on the dual representation of dynamic concave utilities defined on diverse linear spaces of unbounded endowments via the solutions of  BSDEs with unbounded terminal values and attainability of the infimum, and give remarks and examples { (\cref{ex:3.3} and \cref{rmk:3.4})} to illustrate the theorem. Finally in Section 4, we  prove \cref{thm:3.1} with two auxiliary propositions (\cref{pro:4.1,pro:4.3}).\vspace{0.2cm}

\section{Preliminaries}
\label{sec:2-Preliminaries}
\setcounter{equation}{0}

First of all, we introduce some notations and spaces used in this paper. Let $\R_+:=[0,+\infty)$, and for $a,b\in \R$, let $a\wedge b:=\min\{a,b\}$, $a^+:=\max\{a,0\}$ and $a^-:=-\min\{a,0\}$. Let ${\bf 1}_{A}(x)$ represent the indicator function of set $A$. For a convex function $f:\R^{d}\rightarrow\R$, denote by $\partial f(z_0)$ its subdifferential at $z_0$, which is the non-empty convex compact set of elements $u\in\R^d$ such that for each $z\in\R^d$,
$f(z)-f(z_0)\geq u\cdot (z-z_0)$. Denote by $\Sigma_T$ the set of all $(\F_t)$-stopping times $\tau$ taking values in $\T$. For an $(\F_t)$-adapted scalar process $(X_t)_{t\in\T}$, we say that it belongs to class (D) if the family of random variables $\{X_\tau: \tau\in \Sigma_T\}$ is uniformly integrable. Throughout the whole paper, we will be given an $(\F_t)$-progressively measurable nonnegative scalar process $(h_t)_{t\in \T}$, and four positive constants $\gamma,\lambda,c>0$ and $\alpha\in (1,2)$. Let $\alpha^*>2$ denote the conjugate of $\alpha$, i.e., $1/\alpha+1/\alpha^*=1$.\vspace{0.2cm}

For $p,\mu>0$ and $t\in \T$, we denote by $L^p(\F_t)$ the set of $\F_t$-measurable scalar random variables $\eta$ such that $\E[|\eta|^p]<+\infty$, and define the following three spaces of $\F_t$-measurable scalar random variables:
$$
L(\ln L)^p (\F_t):=\left\{\eta\in \F_t\left| \E\left[|\eta|(\ln(1+|\eta|))^p\right]<+\infty\right.\right\},
$$
$$
L\exp[\mu(\ln L)^p](\F_t):=\left\{\eta\in \F_t\left| \E\left[|\eta|\exp{\left(\mu (\ln(1+|\eta|))^p\right)}\right]<+\infty\right.\right\}
$$
and
$$
\exp(\mu L^p)(\F_t):=\left\{\eta\in \F_t\left| \E\left[\exp{\left(\mu |\eta|^p\right)}\right]<+\infty\right.\right\}.
$$
It is not hard to verify that these spaces become smaller when the parameter $\mu$ or $p$ increases, and for each $\mu,\bar\mu,r>0$ and $0<p<1<q$,
$$
 L^\infty(\F_t)\subset \exp(\mu L^r)(\F_t)\subset L\exp[\bar\mu(\ln L)^q](\F_t) \subset L\exp[\mu \ln L](\F_t)=L^{1+\mu}(\F_t)\vspace{-0.1cm}
$$
and
$$
L^q(\F_t)\subset L\exp[\mu(\ln L)^p](\F_t)\subset L(\ln L)^r(\F_t)\subset L^1(\F_t).
$$
For each $p,\mu>0$ and $0<\bar p\leq 1<\tilde p$, it is clear that the following spaces
$$
 L\exp[\mu(\ln L)^{\bar p}](\F_t),\ \ \bigcap\limits_{\bar\mu>0}L\exp[\bar\mu(\ln L)^{\bar p}](\F_t), \  \bigcup\limits_{\bar\mu>0} L\exp[\bar\mu(\ln L)^{\tilde p}](\F_t)\ \ {\rm and}\ \ \bigcap\limits_{\bar\mu>0}\exp(\bar\mu L^p)(\F_t)
$$
are all linear spaces containing $L^\infty(\F_t)$. Note that under the conditions without causing confusion, the $\sigma$-algebra $(\F_T)$ is usually omitted in these notations on the spaces of random variables. We would like to introduce the following practical examples of unbounded endowments in a financial market.

\begin{ex}\label{ex:2.1}
Let $(X_t)_{t\in\T}$ be the unique adapted solution of the following SDE:
$$
\dif X_t=b(t,X_t)\dif t+\sigma(t,X_t)\cdot \dif B_t,\ \ t\in\T;\ \ X_0=x_0,
$$
where $x_0\in \R$ is a constant and both $b(t,x):\T\times\R\To \R$ and $\sigma(t,x):\T\times\R\To \R^d$ are measurable functions such that for each $x_1,x_2\in { \R}$ and $t\in\T$, we have
$$
|b(t,x_1)-b(t,x_2)|+|\sigma(t,x_1)-\sigma(t,x_2)|\leq c|x_1-x_2|\ \ {\rm and}\ \ |b(t,0)|+|\sigma(t,0)|\leq c.
$$
Consider the endowment $\eta:=X_T$. By the classical theory of SDEs we know that
$$
\eta\in L^2\subset \bigcap\limits_{\mu>0}L\exp[\mu(\ln L)^{1\over 2}].
$$
Moreover, if it is also supposed that $|\sigma(t,x)|\leq c$ for each $(t,x)\in \T\times\R$, then the argument in \citet[page 563]{BriandHu2008PTRF} implies that there exists two positive constants $c_1$ and $c_2$ depending only on $(c,T)$ such that
$$
\E\left[\exp\left(c_1\sup_{t\in\T}|X_t|^2\right)\right]\leq c_2\exp(c_2|x_0|^2),\vspace{0.1cm}
$$
which yields that for each $\lambda>0$,
$$\eta:=X_T\in \bigcap\limits_{\mu>0}\exp(\mu L)\subset \bigcap_{\mu>0}\exp(\mu L^{\frac{2}{\alpha^*}})\subset \bigcup_{\mu>0}L\exp[\mu(\ln L)^{1+2\lambda}]\subset L^2.$$
Finally, if we let $d=1$, $b(t,x):=bx$ and $\sigma(t,x):=\sigma x$ for two positive constants $b$ and $\sigma$, then we have
$$
X_t=x_0\exp\left(bt-\frac{1}{2}\sigma^2+\sigma B_t\right),\ \ t\in\T.
$$
Thus, for a typical European call option $\eta:=(X_T-K)^+$ with the strike price $K>0$, we have
$$
\eta:=(X_T-K)^+\in \bigcup_{\mu>0}L\exp[\mu(\ln L)^{1+2\lambda}], \quad  \lambda\in (0, 1/2].\vspace{0.2cm}
$$
\end{ex}

Finally, let us recall recent results on scalar BSDEs. Consider the following scalar BSDE
\begin{equation}\label{eq:2.1}
  Y_t=\xi-\int_t^T g(s,Z_s){\rm d}s+\int_t^T Z_s\cdot {\rm d}B_s, \ \ t\in\T,
\end{equation}
where $\xi$ is called the terminal value, which is an $\F_T$-measurable scalar random variable, the random function $g(\omega, t, z):\Omega\times\T\times\R^d \to \R$, which is $(\F_t)$-progressively measurable for each $z\in \R^d$, is called the generator of \eqref{eq:2.1}, and the pair of $(\F_t)$-progressively measurable processes $(Y_t,Z_t)_{t\in\T}$ taking values in $\R\times\R^d$ is called an adapted solution of \eqref{eq:2.1} if $\ps$, $t\mapsto Y_t$ is continuous, $t\mapsto |g(t,Z_t)|+|Z_t|^2$ is integrable, and \eqref{eq:2.1} holds. Furthermore, we introduce the following assumptions on the generator $g$, recalling that $h_\cdot\geq 0$, $\gamma>0$, $\alpha\in (1,2)$ and $\lambda>0$.

\begin{itemize}
\item [(H0)] $\as$, $g(\omega,t,\cdot)$ is convex.\vspace{-0.2cm}
\item [(H1)] $g$ has a quadratic growth in $z$, i.e., $\as$, $\RE z\in \R^d$,
$|g(\omega,t,z)|\leq h_t(\omega)+\frac{\gamma}{2}|z|^2$.\vspace{-0.2cm}
\item [(H2)] $g$ has a sub-quadratic growth in $z$, i.e., $\as$, $\RE z\in \R^d$, $|g(\omega,t,z)|\leq h_t(\omega)+\gamma |z|^\alpha$.\vspace{-0.2cm}
\item [(H3)] $g$ has a super-linear growth in $z$, i.e., $\as$, $\RE z\in \R^d$, $|g(\omega,t,z)|\leq h_t(\omega)+\gamma |z|\left(\ln(e+|z|)\right)^{\lambda}$.\vspace{-0.2cm}
\item [(H4)] $g$ has a linear growth in $z$, i.e., $\as$, $\RE z\in \R^d$, $|g(\omega,t,z)|\leq h_t(\omega)+\gamma |z|$.
\end{itemize}

We remark that (H1)-(H4) are alternative sets of assumptions on the generator $g$.

\begin{rmk}\label{rmk:2.1}
Both assumptions (H0) and (H4) yield that $\as$, for each $\theta\in (0,1)$ and each $z_1,z_2\in \R^d$, we have\vspace{-0.2cm}
$$
\begin{array}{l}
g(\omega,t,z_1)=\Dis g\left(\omega,t,\theta z_2+(1-\theta)\frac{z_1-\theta z_2}{1-\theta}\right)\leq \Dis \theta g(\omega,t,z_2)+(1-\theta)\left(h_t(\omega)+\gamma \frac{|z_1-\theta z_2|}{1-\theta}\right)\\
\hspace{1.5cm} =\Dis \theta g(\omega,t,z_2)+(1-\theta) h_t(\omega)+\gamma |z_1-\theta z_2|,
\end{array}
$$
and then, by letting $\theta\To 1$ in the last inequality, we have $|g(\omega,t,z_1)-g(\omega,t,z_2)|\leq \gamma |z_1-z_2|$, i.e., the generator $g$ is Lipschitz with respect to $z$.
\end{rmk}

\begin{pro}\label{pro:2.2}
Set $\bar\xi:=|\xi|+\int_0^T h_t\,{\rm d}t$ and $Y^*:=\sup_{t\in \T}|Y_t|$. We have the following four assertions.
\begin{itemize}
\item [(i)] Assume that $\bar\xi\in \bigcap_{\mu>0}\exp(\mu L)$ and the generator $g$ satisfies assumptions (H0) and (H1). Then, BSDE \eqref{eq:2.1} admits a unique adapted solution $(Y_t,Z_t)_{t\in\T}$ such that $Y^*\in \bigcap_{\mu>0}\exp(\mu L)$.

\item [(ii)] Assume that $\bar\xi\in \bigcap_{\mu>0}\exp(\mu L^{\frac{2}{\alpha^*}})$ and the generator $g$ satisfies assumptions (H0) and (H2). Then, BSDE \eqref{eq:2.1} admits a unique adapted solution $(Y_t,Z_t)_{t\in\T}$ such that $Y^*\in \bigcap_{\mu>0}\exp(\mu L^{\frac{2}{\alpha^*}})$.

\item [(iii)] Assume that $\bar\xi\in \bigcap_{\mu>0}L\exp[\mu(\ln L)^{({\frac{1}{2}+\lambda})\vee (2\lambda)}]$ and the generator $g$ satisfies (H0) and (H3). Then, BSDE \eqref{eq:2.1} admits a unique adapted solution $(Y_t,Z_t)_{t\in\T}$ such that for each $\mu>0$,
    the process $(|Y_t|\exp(\mu(\ln(1+|Y_t|))^{({\frac{1}{2}+\lambda})\vee (2\lambda)}))_{t\in\T}$ belongs to class (D). Moreover, if there exists a $\bar\mu>0$ such that $\bar\xi\in L\exp[\bar\mu(\ln L)^{1+2\lambda}]$, then for any positive constant $\tilde\mu<\bar \mu$, we have $Y^*\in L\exp[\tilde\mu(\ln L)^{1+2\lambda}]$.

\item [(iv)] Assume that $\bar\xi\in L\exp[\mu(\ln L)^{1\over 2}]$ for $\mu>\mu_0:=\gamma\sqrt{2T}$ and the generator $g$ satisfies (H0) and (H4). Then, BSDE \eqref{eq:2.1} admits a unique adapted solution $(Y_t,Z_t)_{t\in\T}$ such that the process $|Y_t|\exp(\tilde\mu\sqrt{\ln(1+|Y_t|)}\, ), \ t\in\T, $ belongs to class (D)  for some $\tilde\mu>0$. Moreover, if $\bar\xi\in \bigcap_{\bar\mu>0}L\exp[\bar\mu(\ln L)^{1\over 2}]$, then the process $|Y_t|\exp(\bar\mu\sqrt{\ln(1+|Y_t|)}), \, t\in\T, $ belongs to class (D)  for any $\bar\mu>0$.
\end{itemize}
\end{pro}

\begin{proof}
Assertion (i) is a direct consequence of   \citet[Corollary 6]{BriandHu2008PTRF}, and assertion (ii) follows immediately from~\citet[Theorem 3.1]{FanHu2021SPA}. Furthermore, according to  \citet[Theorem 2.4 ]{FanHuTang2023SPA} together with Doob's maximal inequality for martingales, we easily derive assertion (iii). Finally, in view of \cref{rmk:2.1}, assertion (iv) follows from  \citet[Theorem 3.1]{HuTang2018ECP} and \citet[Theorem 2.5]{BuckdahnHuTang2018ECP}. The reader is also referred to \citet{FanHu2019ECP}, \citet{OKimPak2021CRM} and \citet{FanHuTang2026MF} for more details. The proof is complete.
\end{proof}

\section{Main result}
\label{sec:3-Statement}
\setcounter{equation}{0}

\subsection{Statement of the main result\vspace{0.2cm}}

In the rest of this paper, we always assume that the random core function
$$f(\omega,t,q):\Omega\times\T\times\R^d\To \R_+\cup \{+\infty\}$$
is $(\F_t)$-progressively measurable for each $q\in \R^d$, and that $f$ satisfies the following assumption:
\begin{itemize}
\item [(A0)] $\as$, $f(\omega,t,\cdot)$ is LSC and convex,
and there exists an $(\F_t)$-progressively measurable $\R^d$-valued process $(\bar q_t)_{t\in \T}$ such that $\as$, $f\left(\omega,t, \bar q_t(\omega)\right)=0$.
\end{itemize}
In particular, if $\as$, $f(\omega, t,\cdot)$ is a convex function taking values in $\R_+$ and $f(\omega,t,0)\equiv 0$, then assumption (A0) is naturally satisfied.\vspace{0.2cm}

For each endowment $\xi\in \acal(\F_T)$ and random core function $f$, we define the following process space:
\begin{equation}\label{eq:3.1}
\begin{array}{ll}
\qcal(\xi,f):=\bigg\{&\Dis (q_t)_{t\in[0,T]}~\text{is an } (\F_t)\text{-progressively measurable~} \R^d \text{-valued process}:\\
&\Dis \int_{0}^{T}|q_s|^2\dif s<+\infty~~ \ps, ~\E_{\Q^q}\Big[|\xi|+\int_{0}^{T} f(s,q_s)\,\dif s\Big]<+\infty,\\
&\Dis \text{with} ~L_t^q:=\exp\Big(\int_{0}^{t}q_s\cdot \dif B_s-\frac{1}{2}\int_{0}^{t}|q_s|^2\dif s\Big), \ t\in[0,T],\\
&\Dis \text{being a uniformly integrable martingale, and} \ \frac{\dif \Q^q}{\dif \mathbb{P}}:=L_T^q \ \ \bigg\}.
\end{array}
\end{equation}
In the present paper, we  aim to study that under what conditions on the endowment $\xi$ and the random core function $f$, the space $\qcal(\xi,f)$ is nonempty, the following time-parameterized operator
\begin{equation}\label{eq:3.2}
U_t(\xi):=\essinf\limits_{q\in \qcal(\xi,f)}\E_{\Q^q}\left[\left. \xi+\int_t^T f(s,q_s){\rm d}s\right|\F_t\right],\ \ t\in \T,
\end{equation}
is well-defined and admits a dual representation via the adapted solution of the following BSDE
\begin{equation}\label{eq:3.3}
  Y_t=\xi-\int_t^T g(s,Z_s){\rm d}s+\int_t^T Z_s\cdot {\rm d}B_s, \ \ t\in\T,
\end{equation}
where the generator $g$ of BSDE \eqref{eq:3.3} is the Fenchel-Legendre transform of $f$, i.e.,
\begin{equation}\label{eq:3.4}
g(\omega,t,z):=\sup_{q\in \R^d}(z\cdot q-f(\omega,t,q)),\ \ \RE (\omega,t,z)\in \Omega\tim \T\tim \R^d,
\end{equation}
the infimum in \eqref{eq:3.2} is attainable, and the operator $\{U_t(\cdot),\ t\in\T\}$ defined via \eqref{eq:3.2} constitutes a dynamic concave utility defined on some linear spaces bigger than $L^\infty(\F_T)$. It is clear from \eqref{eq:3.4} and (A0) that $\as$, $g(\omega,t,\cdot)$ is a convex function defined on $\R^d$, $g(\omega,t,0)=0$ and
\begin{equation}\label{eq:3.5}
g(\omega,t,z)\geq z\cdot \bar q_t(\omega)-f(\omega,t,\bar q_t(\omega)) =z\cdot \bar q_t(\omega),\ \ \RE z\in \R^d.\vspace{0.2cm}
\end{equation}

Furthermore, let us introduce the following assumptions on the core function $f$.

\begin{itemize}
\item [(A1)] $\as$, $\RE q\in \R^d$, we have $f(\omega,t,q)\geq \frac{1}{2\gamma} |q|^2$.

\item [(A2)] $\as$, $\RE q\in \R^d$, we have $f(\omega,t,q)\geq {\gamma^{-\frac{1}{\alpha-1}}} |q|^{\alpha^*}$.

\item [(A3)] $\as$, $\RE q\in \R^d$, we have $f(\omega,t,q)\geq c\exp\left(2{\gamma^{-\frac{1}{\lambda}}}|q|^{\frac{1}{\lambda}}\right)-c$.

\item [(A4)] $\as$, $\RE q\in \R^d$, if $|q|>\gamma$, then $f(\omega,t,q)=+\infty$.\vspace{0.2cm}
\end{itemize}

\begin{rmk}\label{rmk:3.0} Any of Assumptions (A1)-(A4)  yields that the dual function $g$ of $f$ satisfies that $\as$, $g(\omega,t,z)<+\infty$ for each $z\in \R^d$. Furthermore, in light of \eqref{eq:3.5}, it can be directly verified that if both (A0) and anyone of (A1)-(A3) are satisfied, then $\as$, $\bar q_t(\omega)=0$ and thus $g(\omega,t,z)\geq 0$ for $z\in \R^d$, and that if both (A0) and (A4) are satisfied, then $\as$, $|\bar q_t(\omega)|\leq \gamma$ and thus $g(\omega,t,z)\geq -\gamma |z|$ for $z\in \R^d$.
\vspace{0.2cm}
\end{rmk}

The main result of the paper is stated as follows.

\begin{thm}\label{thm:3.1}
Denote $Y^*:=\sup_{t\in \T}|Y_t|$. We have the following assertions.
\begin{itemize}
\item [(i)] Assume that the core function $f$ satisfies (A0) and (A1). Then, the generator $g$ defined in \eqref{eq:3.4} satisfies (H0) and (H1) with $h_\cdot=0$. Moreover, for each $\xi\in \bigcap_{\mu>0}\exp(\mu L)$, BSDE \eqref{eq:3.3} admits a unique adapted solution $(Y_t,Z_t)_{t\in\T}$ such that $Y^*\in \bigcap_{\mu>0}\exp(\mu L)$, $(\bar q_t)_{t\in\T}\in \qcal(\xi,f)$, the process $(U_t(\xi))_{t\in\T}$ defined via \eqref{eq:3.2} is well-defined and can be represented as the preceding value process $(Y_t)_{t\in\T}$, and the infimum in \eqref{eq:3.2} is achieved at $q_s\in \partial g(s,Z_s)$, $s\in \T$. Consequently, the operator process $\{U_t(\cdot),\ t\in\T\}$ is a dynamic concave utility on $\bigcap_{\mu>0} \exp(\mu L)$.

\item [(ii)] Assertion (i) remains true when (A1) and (H1) are  replaced respectively by (A2) and (H2) with $h_\cdot=0$, and the space $\bigcap_{\mu>0}\exp(\mu L)$ is replaced by $\bigcap_{\mu>0}\exp(\mu L^{\frac{2}{\alpha^*}})$.

\item [(iii)] Assertion (i) remains true when (A1) and (H1) are replaced respectively by (A3) and (H3) with $h_\cdot=C_{c,\gamma,\lambda}$ being a positive constant depending only on $(c,\gamma,\lambda)$, and the space $\bigcap_{\mu>0}\exp(\mu L)$ is replaced by $\bigcup_{\mu>0}L\exp[\mu(\ln L)^{1+2\lambda}]$.

\item [(iv)] Assertion (i) remains true when (A1) and (H1) are  replaced respectively by (A4) and (H4) with $h_\cdot=0$, the space $\bigcap_{\mu>0}\exp(\mu L)$ is replaced by $\bigcap_{\bar\mu>0}L\exp[\bar\mu(\ln L)^{1\over 2}]$, and the condition $Y^*\in \bigcap_{\mu>0}\exp(\mu L)$ is replaced by the requirement that
    $$
    {\rm for\ each}\ \bar\mu>0,\ {\rm the\ process}\ (|Y_t|\exp(\bar\mu\sqrt{\ln(1+|Y_t|)}))_{t\in\T}\ {\rm belongs\ to\ class\ (D)}.
    $$
\end{itemize}
\end{thm}

\citet{BarrieuElKaroui2009Book} and \citet{KloppelSchweizer2007MF} linked the exponential utility indifference pricing of a contingent claim to the exponential utility maximization problem, and \citet{Faidi2017IJYAF} and \citet{HuImkeller2005AAP} further connected the exponential utility maximization problem to BSDEs with quadratically growing generators appearing in (i) of \cref{pro:2.2} and (i) of \cref{thm:3.1}. In the same spirit, it is natural to apply the other assertions of \cref{pro:2.2} and \cref{thm:3.1} to  the utility indifference pricing of a contingent claim.\vspace{0.2cm}

For the narrative clarity, we depict some key points of \cref{thm:3.1} in the following Table 1: Some key points in \cref{thm:3.1}.\vspace{-0.2cm}

\begin{center}
\captionof{table}{Some key points in \cref{thm:3.1}}\vspace{0.2cm}
{\small
\begin{tabular}
{
|>{\centering\arraybackslash}m{1.1cm}
|>{\centering\arraybackslash}m{3.0cm}
|>{\centering\arraybackslash}m{4.5cm}
|>{\centering\arraybackslash}m{3.3cm}
|>{\centering\arraybackslash}m{1.7cm}|}

\hline

Case of Th 3.1 & Lower bound of $f$ & Growth condition of $g$ & Space for $\xi$ & Underlying reference in Prop. 2.3\\

\hline

\parbox[c][1.5cm]{1.2cm}{\centering (i)} & $f(t,q)\geq \frac{1}{2\gamma} |q|^2$ & $|g(t,z)|\leq h_t+\frac{\gamma}{2}|z|^2$ & $\bigcap_{\mu>0}\exp(\mu L)$ &
(i) \\

\hline

\parbox[c][1.5cm]{1.2cm}{\centering (ii)}  & $f(t,q)\geq {\gamma^{-\frac{1}{\alpha-1}}} |q|^{\alpha^*}$  & $|g(t,z)|\leq h_t+\gamma |z|^\alpha$ & $\bigcap_{\mu>0}\exp(\mu L^{\frac{2}{\alpha^*}})$ &
(ii) \\

\hline

\parbox[c][1.5cm]{1.2cm}{\centering (iii)}  & $f(t,q)\geq c\exp\left(2{\gamma^{-\frac{1}{\lambda}}}|q|^{\frac{1}{\lambda}}\right)-c$ & $|g(t,z)|\leq h_t+\gamma |z|\left(\ln(e+|z|)\right)^{\lambda}$ & $\bigcup_{\mu>0}L\exp[\mu(\ln L)^{1+2\lambda}]$ & (iii) \\

\hline

\parbox[c][1.5cm]{1.2cm}{\centering (iv)}  &  $f(t,q)=+\infty$ for $|q|>\gamma$ & $|g(t,z)|\leq h_t+\gamma |z|$ & $\bigcap_{\bar\mu>0}L\exp[\bar\mu(\ln L)^{1\over 2}]$ & (iv) \\

\hline
\end{tabular}
}
\end{center}
\vspace{0.2cm}

\subsection{Examples and remarks\vspace{0.2cm}}

In this subsection, we are going to present some examples and remarks to illustrate \cref{thm:3.1} and its applications.

\begin{ex}\label{ex:3.3}
We present the following several examples to illustrate \cref{thm:3.1}. We note that for simplicity, some of the constants $(\gamma,\lambda,c,\alpha)$ are taken as concrete values in these examples.
\begin{itemize}
\item [(i)] Let $(\bar q_t)_{t\in \T}$ be an $(\F_t)$-progressively measurable $\R^d$-valued process such that $\ |\bar q_t|\leq \gamma, \as$.  Define the core function $f$ as follows: $\RE (\omega,t,q)\in \Omega\tim\T\tim\R^d$, $f(\omega, t,q):=0$ for $q=\bar q_t(\omega)$ and  $f(\omega, t,q):=+\infty$, otherwise.  Then  the function $f$ satisfies (A0) and (A4). If
\begin{equation}\label{eq:3.6}
\E_{\Q^{\bar q}}\left[|\xi|\right]<+\infty,\vspace{-0.1cm}
\end{equation}
then
\begin{equation}\label{eq:3.7}
U_t(\xi):=\essinf\limits_{q\in \qcal(\xi,f)}\E_{\Q^q}\left[\left. \xi+\int_t^T f(s,q_s){\rm d}s\right|\F_t\right]=\E_{\Q^{\bar q}}\left[\xi|\F_t\right],\ \ t\in\T.
\end{equation}
On the other hand, the function $f$ has the following convex conjugate:
$$
\RE (\omega,t,z)\in \Omega\tim\T\tim\R^d,\ \  g(\omega,t,z):=\sup_{q\in \R^d}(z\cdot q-f(\omega,t,q))=\bar q_t(\omega)\cdot z.
$$
From Girsanov's theorem, we know that when \eqref{eq:3.6} is satisfied, the process $(U_t(\xi))_{t\in\T}$ in \eqref{eq:3.7} is just the value process $Y$ in the unique adapted solution $(Y_t,Z_t)_{t\in\T}$ of BSDE \eqref{eq:3.3} with this generator $g$ { such that $(Y_t)_{t\in\T}$ belongs to class (D) under $\Q^{\bar q}$.} From \cref{thm:3.1} (iv), we see  that for $\xi\in \bigcap_{\mu>0}L\exp[\mu(\ln L)^{1\over 2}]$,  the process $(U_t(\xi))_{t\in\T}$ in \eqref{eq:3.7} is just the value process $Y$ in the unique adapted solution $(Y_t,Z_t)_{t\in\T}$ of BSDE \eqref{eq:3.3} with this $g$ such that for each $\bar\mu>0$,
    $$
    {\rm the\ process}\ (|Y_t|\exp(\bar\mu\sqrt{\ln(1+|Y_t|)}))_{t\in\T}\ {\rm belongs\ to\ class\ (D)}.
    $$

\item [(ii)] Define the core function $f$ as follows: $\RE (\omega,t,q)\in \Omega\tim\T\tim\R^d$, $f(\omega, t,q):=0$ when $|q|\leq \gamma$; otherwise, $f(\omega, t,q):=+\infty$. Then $f$ satisfies (A0) and (A4) with $\bar q_\cdot=0$. For $\xi\in L^2$, we have
\begin{equation}\label{eq:3.8}
U_t(\xi):=\essinf\limits_{q\in \qcal(\xi,f)}\E_{\Q^q}\left[\left. \xi+\int_t^T f(s,q_s){\rm d}s\right|\F_t\right]=\essinf_{q_\cdot\in\R^d:|q_\cdot|\leq \gamma}\E_{\Q^{\bar q}}\left[\xi |\F_t\right],\ \ t\in\T.
\end{equation}
On the other hand, the function $f$ has the following convex conjugate:
$$
\RE (\omega,t,z)\in \Omega\tim\T\tim\R^d,\ \ g(\omega,t,z):=\sup_{q\in \R^d}(z\cdot q-f(\omega,t,q))=\gamma |z|.
$$
From Lemma 3 of \citet{ChenPeng2000SPL}, we see that for each $\xi\in L^2$, the process $(U_t(\xi))_{t\in\T}$ in \eqref{eq:3.8} is just the value process $(Y_t)_{t\in\T}$ in the unique adapted solution $(Y_t,Z_t)_{t\in\T}$ of BSDE \eqref{eq:3.3} with this generator $g$ such that $\E[\sup_{t\in\T}|Y_t|^2+\int_0^T |Z_t|^2 {\rm d}t]<+\infty$. From \cref{thm:3.1} (iv), we see that for each $\xi\in \bigcap_{\mu>0}L\exp[\mu(\ln L)^{1\over 2}]$, the process $(U_t(\xi))_{t\in\T}$ in \eqref{eq:3.8} is just the value process $Y$ in the unique adapted solution $(Y_t,Z_t)_{t\in\T}$ of BSDE \eqref{eq:3.3} with this generator $g$ such that for each $\bar\mu>0$,
${\rm the\ process}\ (|Y_t|\exp(\bar\mu\sqrt{\ln(1+|Y_t|)}))_{t\in\T}\ {\rm belongs\ to\ class\ (D)}.$

\item [(iii)] Define the core function $f$ as follows: $\RE (\omega,t,q)\in \Omega\tim\T\tim\R^d$, $f(\omega, t,q):=\frac{1}{2\gamma}|q|^2$. Then $f$ satisfies assumptions (A0) and (A1) with $\bar q_\cdot=0$, and it has the following convex conjugate: $\RE (\omega,t,z)\in \Omega\tim\T\tim\R^d$,
$$
g(\omega,t,z):=\sup_{q\in \R^d}(z\cdot q-f(\omega,t,q))=\frac{\gamma}{2} |z|^2.
$$
From \citet{BriandHu2006PTRF}, we see that for $\xi\in \exp(\gamma L)$, BSDE \eqref{eq:3.3} with the above generator $g$ admits a unique adapted solution $(Y_t,Z_t)_{t\in\T}$ such that $\{\exp(\gamma |Y_t|)\}_{t\in\T}$ belongs to class (D), and the process $-Y$ is the well-known dynamic entropic risk measure of $-\xi$ (see \citet{DetlefsenScandolo2005FS}): $Y_t=-\frac{1}{\gamma}\ln \E\left[\exp(-\gamma \xi) |\F_t\right],\ t\in\T$. Thus, from \cref{thm:3.1} (i), we have
$$
U_t(\xi):=\essinf\limits_{q\in \qcal(\xi,f)}\E_{\Q^q}\left[\left. \xi+\frac{1}{2\gamma}\int_t^T |q_s|^2{\rm d}s\right|\F_t\right]=-\frac{1}{\gamma}\ln \E\left[\exp(-\gamma \xi) |\F_t\right],\ \ t\in\T
$$
for $\xi\in \bigcap_{\mu>0}\exp(\mu L)$.

\item [(iv)] Define the core function $f$ as follows: $\RE (\omega,t,q)\in \Omega\tim\T\tim\R^d$, $f(\omega, t,q):=\frac{1}{2}|q|^2$ for $|q|\leq \gamma$; otherwise, $f(\omega, t,q):=+\infty$. Then $f$ satisfies (A0) and (A4) with $\bar q_\cdot=0$, and it has the following convex conjugate: $\RE (\omega,t,z)\in \Omega\tim\T\tim\R^d$,
$$
g(\omega,t,z):=\sup_{q\in \R^d}(z\cdot q-f(\omega,t,q))=\left\{
\begin{array}{l}
\Dis \frac{1}{2}|z|^2,\ \ \ \ \ \ \ \ \ ~|z|\leq \gamma;\\
\Dis \gamma |z|-\frac{1}{2}\gamma^2,\ \ \ |z|>\gamma,\\
\end{array}
\right.
$$
which  satisfies assumptions (H0) and (H4). Thus, if $\xi\in \bigcap_{\mu>0}L\exp[\mu(\ln L)^{1\over 2}]$, the conclusions in  \cref{thm:3.1} (iv) can be applied.

\item [(v)] Define the core function $f$ as follows:  $\RE (\omega,t,q)\in \Omega\tim\T\tim\R^d$, $f(\omega,t,q):=e^{|q|}-{\color{red}1}$. Then $f$ satisfies assumptions (A0) and (A3) with $\bar q_\cdot=0$ and $(c,\gamma,\lambda)=(1,2,1)$, and it has the following convex conjugate:  $\RE (\omega,t,z)\in \Omega\tim\T\tim\R^d$,
$$
g(\omega,t,z):=\sup_{q\in \R^d}(z\cdot q-f(\omega,t,q))=|z|(\ln|z|-1)-{\color{red}1},
$$
which satisfies assumptions (H0) and (H3). Thus, if $\xi\in \bigcup_{\mu>0}L\exp[\mu(\ln L)^3]$, then the conclusions in \cref{thm:3.1} (iii) can be applied.

\item [(vi)] Define the core function $f$ as follows:  $\RE (\omega,t,q)\in \Omega\tim\T\tim\R^d$, $f(\omega,t,q):=\frac{1}{4}|q|^4$. Then $f$ satisfies assumptions (A0) and (A2) with  $\bar q_\cdot=0$ and $(\alpha,\alpha^*,\gamma)=(\frac{4}{3},4,\sqrt[3]{4})$, and it has the following convex conjugate: $\RE (\omega,t,z)\in \Omega\tim\T\tim\R^d$,
$$
g(\omega,t,z):=\sup_{q\in \R^d}(z\cdot q-f(\omega,t,q))=\frac{3}{4}|z|^{\frac{4}{3}},
$$
which satisfies assumptions (H0) and (H2). Thus, if $\xi\in \bigcap_{\mu>0}\exp(\mu L^{\frac{1}{2}})$, then the conclusions in \cref{thm:3.1} (ii) can be applied.

\item [(vii)] Let $d=1$ and define the core function $f$ as follows:  $\RE (\omega,t,q)\in \Omega\tim\T\tim\R$,
$$
f(\omega, t,q):=\left\{
\begin{array}{l}
+\infty,\ \ \ \ q<0;\\
\frac{1}{2}q,\ \ \ \ \ \ 0\leq q\leq 2;\\
\frac{1}{4}q^2,\ \ \ \ \ q>2.
\end{array}
\right.
$$
Then $f$ satisfies assumptions (A0) and (A1) with $\bar q_\cdot=0$ and $\gamma=2$, and it has the following convex conjugate:  $\RE (\omega,t,z)\in \Omega\tim\T\tim\R$,
$$
g(\omega,t,z):=\sup_{q\in \R^d}(z\cdot q-f(\omega,t,q))=\left\{
\begin{array}{l}
0,\ \ \ \ \ \ \ \ z<\frac{1}{2};\\
2z-1,\ \ \frac{1}{2}\leq z\leq 1;\\
z^2,\ \ \ \ \ \ \ z\geq 1,
\end{array}
\right.
$$
which satisfies assumptions (H0) and (H1). Thus, if $\xi\in \bigcap_{\mu>0}\exp(\mu L)$, then the conclusions in  \cref{thm:3.1} (i) can be applied.
\end{itemize}
\end{ex}
\vspace{0.2cm}

\begin{rmk}\label{rmk:3.4}
We have the following comments on \cref{thm:3.1}.
\begin{itemize}
\item [(i)] As stated in the introduction, for the special case of $\xi\in L^\infty$, Assertions (i) and (iv) of \cref{thm:3.1} are available,  respectively, in \citet{DelbaenHuBao2011PTRF} when the core function $f$ does not depend on $(\omega,t)$, and in \citet{DelbaenPengRosazza2010FS} when $f(t,0)=0$. They extend the latter to possibly unbounded endowments $\xi$ and more general core functions $f$. On the other hand, Assertions (ii) and (iii) of \cref{thm:3.1} seem to be completely new.

\item [(ii)] As stated in the introduction, when the core function $f$ is independent of $(\omega,t)$ such that $f(0)=0$, Theorems 4.5 and 4.6 in \citet{Drapeau2016AIHPPS} indicates that if the infimum of $U_0(\xi)$ in \eqref{eq:1.5} is attainable within $\qcal_2$  for $\xi\in \acal_0(\F_T)$, the dynamic concave utility on $\acal_0(\F_T)$ has a dual representation as the value process in the unique adapted solution to BSDE \eqref{eq:1.3}. However, generally speaking, it is hard to verify the attainability of the infimum in \eqref{eq:1.5} since the process $L_t^q$ for  $q\in \qcal_2$ is required to be essentially bounded. In \cref{thm:3.1}, in four different cases on $(\xi,f)$,  we use the set $\qcal(\xi,f)$  instead of $\qcal_2$ so that the infimum in \eqref{eq:3.2} is well-defined and attainable.

\item [(iii)] \cref{thm:3.1} illustrates that the dual representation, as the value process in the unique solution of a BSDE, of several class of dynamic concave utilities holds for unbounded endowments and general random core functions, and that the minimizer of the utility can be achieved in suitable set $\qcal(\xi,f)$. Obviously, since the admissible set $\qcal(\xi,f)$ is in general of infinite dimension, it is very difficult to directly compute the dynamic concave utility via the definition \eqref{eq:3.2}. However, \cref{thm:3.1} enables us  to compute the unbounded dynamic concave utilities via  solution of BSDEs by numerical algorithms such as Monte Carlo method; see for example \citet{Chassagneux2016AAP} and \citet{Chessari2023PS} for more details.

\item [(iv)] Our proof of \cref{thm:3.1} is enlightened by \citet{DelbaenHuBao2011PTRF}, \citet{DelbaenPengRosazza2010FS} and in particular by \citet{DelbaenHuRichou2011AIHPPS}, where the Fenchel-Legendre conjugate transform for  convex functions and the de la Vall\'{e}e-Poussin theorem play a crucial role. Moreover, using both inequalities of Young  and Gronwall,  we establish two useful technical propositions to overcome some essential difficulties. See the next section for details.

\item [(v)] In \cref{thm:3.1}, the integrability requirement on $\xi$ does not depend on both constants $\gamma$ and $c$, which is a  convenience to the application of \cref{thm:3.1}. In fact, for a given core function $f$, both constants $\gamma$ and $c$ can be chosen so that one of Assumptions (A1)-(A4) is satisfied for $f$.

\item [(vi)] The space $L^p\ (p>1)$ is smaller than  $\bigcap_{\mu>0}L\exp[\mu(\ln L)^{\frac{1}{2}}]$ and larger than $\bigcup_{\mu>0}L\exp[\mu(\ln L)^{1+2\lambda}]$. It would be interesting to give some appropriate growth conditions on the core function $f$, which are weaker than (A4) and stronger than (A3), such that the dual representation in \cref{thm:3.1} holds for endowment $\xi\in L^2$ or $\xi\in L^p$ for some $p>1$. The existence of such conditions remains an open question.

\item [(vii)] If the generator $g$ of BSDE \eqref{eq:3.3} also depends on the value process and is decreasing in it, then by \eqref{eq:3.3} we can define a cash-subadditive risk measure as in \citet{ElKarouiRavanelli2009MF}. The dynamic cash-subadditive risk measure was studied in \citet{Drapeau2016AIHPPS}, \citet{RosazzaZullino2024PUQR} and \citet{Laeven2023arXiv}. However, unbounded dynamic convex risk measures (concave utilities) with the cash-subadditivity still remains to be studied.
\end{itemize}
\end{rmk}

\section{Proof of \cref{thm:3.1}}
\label{sec:4-Proof}
\setcounter{equation}{0}

\subsection{Two auxiliary propositions\vspace{0.2cm}}

We first establish two auxiliary technical propositions, which are interesting in their own right.\vspace{0.2cm}

First of all, let us recall the known Young inequality. Let $f:\R_+\rightarrow \R_+$ be a strictly increasing continuous function with $f(0)=0$, and $g$ be the inverse function of $f$. Then, we have
\begin{equation}\label{eq:4.1}
xy\leq \int_0^x f(s){\rm d}s+\int_0^y g(s){\rm d}s\leq xf(x)+yg(y),\ \ \RE x,y\geq 0.
\end{equation}
Based on this elementary inequality, we easily prove the following

\begin{pro}\label{pro:4.1}
Let $\mu,\delta>0$ be two arbitrarily given positive constants. Then, we have
\begin{equation}\label{eq:4.2}
xy\leq x\exp\left(\frac{x^\delta}{\mu^\delta}\right)+\mu y(\ln(1+y))^{\frac{1}{\delta}},\ \ \RE x,y\geq 0
\end{equation}
and
\begin{equation}\label{eq:4.3}
xy\leq x\exp\left(\frac{(\ln(1+x))^\delta}{\mu^\delta}\right)+y\exp\left(\mu (\ln(1+y))^{\frac{1}{\delta}}\right),\ \ \RE x,y\geq 0.\vspace{0.1cm}
\end{equation}
Furthermore, for each $q>1$ and $\eps>0$, there exists a constant $C_{q,\eps,\mu,\delta}>0$ depending only on $(q,\eps,\mu,\delta)$ such that
\begin{equation}\label{eq:4.4}
xy\leq \eps \exp\left(\frac{qx^\delta}{\mu^\delta}\right)+\mu y(\ln(1+y))^{\frac{1}{\delta}}+C_{q,\eps,\mu,\delta},\ \ \RE x,y\geq 0
\end{equation}
and in the case of $\delta>1$,
\begin{equation}\label{eq:4.5}
xy\leq \eps\exp\left(\frac{q(\ln(1+x))^\delta}{\mu^\delta}\right)+y\exp\left(\mu (\ln(1+y))^{\frac{1}{\delta}}\right)+C_{q,\eps,\mu,\delta},\ \ \RE x,y\geq 0.
\end{equation}
In particular, we have
\begin{equation}\label{eq:4.6}
xy\leq \mu\exp\left(\frac{x}{\mu}\right)+\mu y\ln(1+y),\ \ \RE x, y\geq 0,\vspace{0.2cm}
\end{equation}
and for each $q>1$, there exists a constant $\bar C_{\mu,q}>0$ depending only on $(\mu,q)$ such that
\begin{equation}\label{eq:4.7}
y\exp(x) \leq \bar C_{\mu,q}\exp\left(\frac{q x^2}{\mu^2}\right)+y\exp\left(\mu \sqrt{\ln(1+y)}\right), \ \ \RE x,y\geq 0.
\end{equation}
In addition, for $\delta>1$, we have
\begin{equation}\label{eq:4.8}
xy\leq \mu x^{\delta}+\frac{1}{\mu^{\frac{1}{\delta-1}}}y^{\delta^*},\ \ \RE x,y\geq 0,
\end{equation}
where $\delta^*:=\delta/(\delta-1)$ is the conjugate of $\delta$, i.e., $1/\delta+1/\delta^*=1$.
\end{pro}

\begin{proof}
Inequalities \eqref{eq:4.2} and \eqref{eq:4.3} follow immediately by picking
$$f(x)=\exp\left(\frac{x^\delta}{\mu^\delta}\right)-1\ \ {\rm and}\ \ f(x)=\exp\left(\frac{(\ln(1+x))^\delta}{\mu^\delta}\right)-1,$$
respectively, in \eqref{eq:4.1}. Observe that for each $q>1$ and $\eps>0$, we have
$$
\lim\limits_{x\To +\infty} \frac{\eps\exp\left(\frac{(q-1)x^\delta}{\mu^\delta}\right)}{x}=+\infty,\ \RE \delta>0
$$
and
$$
\lim\limits_{x\To +\infty} \frac{\eps \exp\left(\frac{(q-1)(\ln(1+x))^\delta}{\mu^\delta}\right)}{x}=+\infty,\ \RE \delta>1.\vspace{0.1cm}
$$
Then, there exists a constant $k_{q,\eps,\mu,\delta}>0$ depending only on $(q,\mu,\delta,\eps)$ such that for each $x\geq k_{q,\eps,\mu,\delta}$,
$$
x\leq \eps\exp\left(\frac{(q-1)x^\delta}{\mu^\delta}\right),\ \RE \delta>0
$$
and
$$
x\leq \eps\exp\left(\frac{(q-1)(\ln(1+x))^\delta}{\mu^\delta}\right),\ \RE \delta>1.\vspace{0.1cm}
$$
Thus, the desired inequalities \eqref{eq:4.4} and \eqref{eq:4.5} follow immediately from \eqref{eq:4.2} and \eqref{eq:4.3} together with the last two inequalities. By picking $f(x)=\exp(\frac{x}{\mu})-1$ in \eqref{eq:4.1}, we have \eqref{eq:4.6}. Furthermore, note that for each $q>1$ and $x\geq 0$, we have
$0\leq \ln(1+\exp(x))-\ln(\exp(x))=\ln(1+\exp(-x))\leq \ln 2$,
$$
\left(\ln(1+\exp(x))\right)^2\leq \left(\ln 2+x\right)^2\leq \left(1+\frac{2}{q-1}\right)(\ln 2)^2+\frac{q+1}{2} x^2
$$
and
$$
\exp(x)\leq \exp\left(\frac{q-1}{2\mu^2}x^2+\frac{\mu^2}{2(q-1)}\right).\vspace{0.1cm}
$$
By letting $\exp(x)$ instead of $x$ and $\delta=2$ in \eqref{eq:4.3}, inequality \eqref{eq:4.7} follows immediately by picking
$$
\bar C_{\mu,q}:=\exp\left(\frac{\mu^2}{2(q-1)}+\frac{(q+1)(\ln 2)^2}{(q-1)\mu^2}\right).
$$
Finally, by picking $f(x)=\mu x^{\delta-1}$ in \eqref{eq:4.1}, we get \eqref{eq:4.8}. The proof is then complete.
\end{proof}

\begin{rmk}\label{rmk:4.2}
We make the following remarks regarding several inequalities in \cref{pro:4.1}.
\begin{itemize}
\item [(i)] Inequality \eqref{eq:4.6} can be compared to the well-known Fenchel inequality:
\begin{equation}\label{eq:4.9}
xy\leq \exp(x)+y(\ln y-1),\ \ \RE x\in\R,\ \RE y>0.
\end{equation}
It is easy to check that the inequality \eqref{eq:4.6} plays the same role as \eqref{eq:4.9} in \cite{DelbaenHuRichou2011AIHPPS} and \cite{DelbaenHuRichou2015DCDS}.

\item [(ii)] Inequality \eqref{eq:4.7} can be compared to the following one
\begin{equation}\label{eq:4.10}
y\exp(x) \leq \exp\left(\frac{x^2}{\mu^2}\right)+\exp(\mu^2)y\exp\left(\mu \sqrt{\ln(1+y)}\right), \ \ \RE x\in \R,\ \RE y\geq 0,\ \RE \mu>0,
\end{equation}
which comes from Lemma 2.4 in \citet{HuTang2018ECP}. Note that the claim of $\mu>\gamma \sqrt{2T}$ for some certain positive constant $\gamma>0$ is equivalent to that of $\mu>q\gamma \sqrt{2T}$ for some suitable constant $q>1$. It is not hard to verify that \eqref{eq:4.7} plays the same role as \eqref{eq:4.10} in \cite{HuTang2018ECP}, \cite{BuckdahnHuTang2018ECP}, \cite{FanHu2019ECP}, \cite{OKimPak2021CRM} and \cite{FanHuTang2023SPA}.

\item [(iii)] Inequality \eqref{eq:4.8} can be compared to the classical Young inequality
\begin{equation}\label{eq:4.11}
xy\leq \frac{1}{\delta} x^{\delta}+\frac{1}{\delta^*}y^{\delta^*},\ \ \RE x,y\geq 0,\ \RE \delta>1.
\end{equation}
It seems to be more convenient to use \eqref{eq:4.8} than \eqref{eq:4.11}.
\end{itemize}
\end{rmk}

By Gronwall's inequality,  we have

\begin{pro}\label{pro:4.3}
Let $(q_t)_{t\in\T}$ {  be} an $(\F_t)$-progressively measurable $\R^d$-valued process such that $\ps$, $\int_0^T |q_s|^2\dif s<+\infty$. Define its stochastic exponential by $L_t^q:=\exp(\int_{0}^{t}q_s\cdot \dif B_s-\frac{1}{2}\int_{0}^{t}| q_s|^2\dif s), \ t\in[0,T]$. Then, we have
\begin{equation}\label{eq:4.12}
\E\left[L^q_{T}\ln(1+L^q_{T})\right]\leq \frac{1}{2}\E\left[\int_0^{T} L^q_t|q_t|^2\dif t\right]+\ln 2
\end{equation}
and
\begin{equation}\label{eq:4.13}
\E\left[L^q_{T}\left[\ln(1+L^q_{T})\right]^{\frac{\alpha^*}{2}}\right]\leq
C_{\alpha,T} \E\left[\int_0^{T} L^q_t|q_t|^{\alpha^*}\dif t\right]+C_{\alpha,T},
\end{equation}
where $C_{\alpha,T}$ is a positive constant depending only on $(\alpha,T)$, { and $\alpha^*$ is the conjugate of $\alpha$}. Furthermore, for each $\mu,\eps>0$, there exists a positive constant $\tilde C_{\mu,\eps,\gamma,\lambda,T}>0$ depending only on $(\mu,\eps,\gamma, \lambda, T)$ such that
\begin{equation}\label{eq:4.14}
\E\left[L^q_{T}\exp\left(\mu \left[\ln(1+L^q_{T})\right]^{\frac{1}{1+2\lambda}}\right)\right]
\leq \Dis \eps\E\left[\int_0^T L_t^q \exp\left(2{\gamma^{-\frac{1}{\lambda}}}|q_t|^{\frac{1}{\lambda}}\right)\dif t\right]+ \tilde C_{\mu,\eps,\gamma,\lambda,T}.\vspace{0.2cm}
\end{equation}
\end{pro}

\begin{proof}
It suffices to consider the case that those terms in the right hand side of \eqref{eq:4.12}-\eqref{eq:4.14} are finite. Note that
\begin{equation}\label{eq:4.15}
\dif L_t^q=L_t^q q_t\cdot \dif B_t,\ \ t\in \T.\vspace{0.2cm}
\end{equation}

We first verify \eqref{eq:4.12}. Define  the test function $l(x)=x\ln(1+x), \ x>0$. Then, for $x>0$, we have
$$
l'(x)=\ln (1+x)+\frac{x}{1+x}>0\ \ {\rm and}\ \ 0<l''(x)=\frac{1}{1+x}+\frac{1}{(1+x)^2}< \frac{1}{x}.
$$
It then follows from It\^{o}'s formula that, in light of \eqref{eq:4.15},
\begin{equation}\label{eq:4.16}
\dif l(L^q_t)=l'(L^q_t)L_t^q q_t\cdot \dif B_t+\frac{1}{2}l''(L^q_t)(L_t^q)^2 |q_t|^2\dif t\leq l'(L^q_t)L_t^q q_t\cdot \dif B_t+\frac{1}{2}L_t^q |q_t|^2\dif t,\ \ t\in \T.
\end{equation}
For each $n\geq 1$, define the following $(\F_t)$-stopping time:
$$
\tau_n:=\inf\{t\geq 0: \int_{0}^{t}\left(l'(L_s^q)L_s^q\right)^2 |q_s|^2\dif s\geq n\}\wedge T.
$$
Since
$$\int_{0}^{T}\left(l'(L_s^q)L_s^q\right)^2 |q_s|^2\dif s\leq \sup_{t\in\T}\left(l'(L_t^q)L_t^q\right)^2
\int_{0}^{T}|q_s|^2\dif s<+\infty\ \ \ps,\vspace{0.1cm}$$
it follows that $\tau_n\To T$ as $n$ tends to infinity. By \eqref{eq:4.16} we deduce that for each $n\geq 1$,
$$
\E[l(L^q_{\tau_n})]-\ln 2\leq \frac{1}{2} \E\left[\int_0^{\tau_n} L_t^q |q_t|^2\dif t\right].
$$
By sending $n$ to infinity in the last inequality and using Fatou's lemma, we can easily obtain \eqref{eq:4.12}.\vspace{0.2cm}

We now prove \eqref{eq:4.13}. Define the test function $\bar l(x)=(e+x)[\ln(e+x)]^{\frac{\alpha^*}{2}}, \ x\in\R_+$. Then, for $x\in\R_+$, we have
$$
\bar l'(x)=[\ln (e+x)]^{\frac{\alpha^*}{2}}+\frac{\alpha^*}{2}[\ln(e+x)]^{\frac{\alpha^*}{2}-1}>0
$$
and
$$
0<\bar l''(x) = \Dis \frac{\alpha^*}{2}[\ln(e+x)]^{\frac{\alpha^*}{2}-1} \frac{1}{e+x}+\frac{\alpha^*(\alpha^*-2)}{4}[\ln(e+x)]^{\frac{\alpha^*}{2}-2}
\frac{1}{e+x}<\Dis \frac{(\alpha^*)^2}{4}\frac{[\ln(e+x)]^{\frac{\alpha^*}{2}-1}}{x}.
$$
Furthermore, for $x,y\in\R_+$,  by Young's inequality we get
$$
[\ln(e+x)]^{\frac{\alpha^*}{2}-1}y^2=\left([\ln(e+x)]^{\frac{\alpha^*}{2}}
\right)^{\frac{\alpha^*-2}{\alpha^*}}\left(y^{\alpha^*}\right)^{\frac{2}{\alpha^*}}\leq \frac{\alpha^*-2}{\alpha^*}[\ln(e+x)]^{\frac{\alpha^*}{2}}+\frac{2}{\alpha^*}y^{\alpha^*}.
$$
It then follows from It\^{o}'s formula that for each $t\in\T$,
\begin{equation}\label{eq:4.17}
\begin{array}{lll}
\dif \bar l(L^q_t)&=& \Dis \bar l'(L^q_t)L_t^q q_t\cdot \dif B_t+\frac{1}{2}\bar l''(L^q_t)(L_t^q)^2 |q_t|^2\dif t\\
&\leq & \Dis \bar l'(L^q_t)L_t^q q_t\cdot \dif B_t+L_t^q \Big(
\frac{1}{8}\alpha^*(\alpha^*-2)[\ln(e+L_t^q )]^{\frac{\alpha^*}{2}}
+\frac{\alpha^*}{4} |q_t|^{\alpha^*}\Big)\dif t\\
&\leq & \Dis \bar l'(L^q_t)L_t^q q_t\cdot \dif B_t+\Big(
\frac{1}{8}\alpha^*(\alpha^*-2)\bar l(L^q_t)
+\frac{\alpha^*}{4}L_t^q |q_t|^{\alpha^*}\Big)\dif t.
\end{array}
\end{equation}
For each $n\geq 1$, define the following $(\F_t)$-stopping time:
$$
\tau_n:=\inf\{t\geq 0: \int_{0}^{t}\left(\bar l'(L_s^q)L_s^q\right)^2 |q_s|^2\dif s\geq n\}\wedge T.
$$
Since
$$\int_{0}^{T}\left(\bar l'(L_s^q)L_s^q\right)^2 |q_s|^2\dif s\leq \sup_{t\in\T}\left(\bar l'(L_t^q)L_t^q\right)^2
\int_{0}^{T}|q_s|^2\dif s<+\infty\ \ \ps,\vspace{0.1cm}$$
it follows that $\tau_n\To T$ as $n$ tends to infinity. By \eqref{eq:4.17}, we  deduce that for each $n\geq 1$ and $t\in \T$,
$$
\begin{array}{lll}
\E[\bar l(L^q_{t\wedge \tau_n})]-e &\leq & \Dis \frac{1}{8}\alpha^*(\alpha^*-2)\E\left[\int_0^{t\wedge \tau_n}\bar l(L^q_{s}) \dif s\right] +\frac{\alpha^*}{4}\E\left[\int_0^{t\wedge \tau_n} L_s^q |q_s|^{\alpha^*}\dif s\right]\\
 &\leq & \Dis \frac{1}{8}\alpha^*(\alpha^*-2)\int_0^t \E[\bar l(L^q_{s\wedge \tau_n})] \dif s +\frac{\alpha^*}{4}\E\left[\int_0^T L_s^q |q_s|^{\alpha^*}\dif s\right].
\end{array}
$$
Applying Gronwall's inequality to the last inequality yields that for each $n\geq 1$,
$$
\E[\bar l(L^q_{t\wedge \tau_n})]\leq \left(\frac{\alpha^*}{4}\E\left[\int_0^T L_s^q |q_s|^{\alpha^*}\dif s\right]+e\right)\exp\left(\frac{1}{8}\alpha^*(\alpha^*-2)t\right),\ \ t\in \T.
$$
By letting $t=T$ and sending $n$ to infinity in the last inequality, and then by using Fatou's lemma, we can obtain the desired inequality \eqref{eq:4.13}.\vspace{0.2cm}

Finally, we verify \eqref{eq:4.14}. Given $\mu>0$ and define the test function $\tilde l(x)=(\tilde k +x)\exp\left(\mu [\ln(\tilde k +x)]^{\frac{1}{1+2\lambda}}\right), \ x\in\R_+$, where $\tilde k$ is a fixed positive constant large enough such that for $x\in\R_+$,
$$
\tilde l'(x)=\tilde l(x)\frac{\frac{\mu}{1+2\lambda}+[\ln(\tilde k +x)]^{\frac{2\lambda}{1+2\lambda}} }{(\tilde k +x)[\ln(\tilde k +x)]^{\frac{2\lambda}{1+2\lambda}}} >0
$$
and
\begin{equation}\label{eq:4.18}
\begin{array}{lll}
0<\tilde l''(x)&=& \Dis \tilde l(x)\frac{\frac{\mu}{1+2\lambda}\left[
\ln(\tilde k +x)+\frac{\mu}{1+2\lambda}[\ln(\tilde k +x)]^{\frac{1}{1+2\lambda}}
-\frac{2\lambda}{1+2\lambda}
\right]}
{(\tilde k +x)^2[\ln(\tilde k +x)]^{\frac{1+4\lambda}{1+2\lambda}}}\vspace{0.2cm}\\
&<& \Dis \frac{2\mu}{1+2\lambda}\frac{\exp\left(\mu [\ln(\tilde k +x)]^{\frac{1}{1+2\lambda}}\right)}
{x [\ln(\tilde k +x)]^{\frac{2\lambda}{1+2\lambda}}}.\vspace{0.1cm}
\end{array}
\end{equation}
Furthermore, it follows from (4.2) with $(\mu,\delta)=(\gamma^2,\frac{1}{2\lambda})$ that there exists a positive constant $\bar k_{\lambda,\gamma}$ depending only on $\lambda$ such that for $x>1$ and $y\geq 0$,
$$
xy\leq y\exp
\left(\gamma^{-\frac{1}{\lambda}}y^{\frac{1}{2\lambda}}\right)+\gamma^2 x[\ln(1+x)]^{2\lambda}\leq 2\gamma^2 x(\ln x)^{2\lambda}+y\exp
\left(\gamma^{-\frac{1}{\lambda}}y^{\frac{1}{2\lambda}}\right)+\bar k_{\lambda,\gamma},
$$
and then for $x,y\in\R_+$ we have, in light of $\tilde k$ being large enough,
\begin{equation}\label{eq:4.19}
\begin{array}{lll}
\Dis \frac{\exp\left(\mu [\ln(\tilde k +x)]^{\frac{1}{1+2\lambda}}\right)}
{[\ln(\tilde k +x)]^{\frac{2\lambda}{1+2\lambda}}}y^2 &\leq & \Dis 2\gamma^2
\frac{\exp\left(\mu [\ln(\tilde k +x)]^{\frac{1}{1+2\lambda}}\right)}
{[\ln(\tilde k +x)]^{\frac{2\lambda}{1+2\lambda}}} \left(\mu [\ln(\tilde k +x)]^{\frac{1}{1+2\lambda}}\right)^{2\lambda}\vspace{0.1cm}\\
&& \Dis \ \ +y^2\exp\left(\gamma^{-\frac{1}{\lambda}}y^{\frac{1}{\lambda}}\right)
+\bar k_{\lambda,\gamma}\vspace{0.1cm}\\
&= & \Dis 2\gamma^2\mu^{2\lambda}\exp\left(\mu [\ln(\tilde k +x)]^{\frac{1}{1+2\lambda}}\right)
+y^2\exp\left(\gamma^{-\frac{1}{\lambda}}y^{\frac{1}{\lambda}}\right)
+\bar k_{\lambda,\gamma}.
\end{array}
\end{equation}
It then follows from It\^{o}'s formula together with \eqref{eq:4.18} and \eqref{eq:4.19} that for each $t\in\T$,
\begin{equation}\label{eq:4.20}
\begin{array}{lll}
\dif \tilde l(L^q_t)&=& \Dis \tilde l'(L^q_t)L_t^q q_t\cdot \dif B_t+\frac{1}{2}\tilde l''(L^q_t)(L_t^q)^2 |q_t|^2\dif t\leq \tilde l'(L^q_t)L_t^q q_t\cdot \dif B_t\vspace{0.1cm}\\
&& \Dis \ \ +\frac{\mu L_t^q}{1+2\lambda}\left(2\gamma^2\mu^{2\lambda}\exp\left(\mu [\ln(k+L_t^q)]^{\frac{1}{1+2\lambda}}\right)
+|q_t|^2\exp\left(\gamma^{-\frac{1}{\lambda}}|q_t|^{\frac{1}{\lambda}}\right)
+\bar k_{\lambda,\gamma}\right)\dif t\vspace{0.1cm}\\
&\leq & \Dis \tilde l'(L^q_t)L_t^q q_t\cdot \dif B_t+\frac{\mu }{1+2\lambda}\left(2\gamma^2\mu^{2\lambda}\tilde l(L_t^q)
+L_t^q |q_t|^2\exp\left(\gamma^{-\frac{1}{\lambda}}|q_t|^{\frac{1}{\lambda}}\right)
+L_t^q \bar k_{\lambda,\gamma}\right)\dif t.
\end{array}
\end{equation}
For $n\geq 1$, define the following $(\F_t)$-stopping time
$$
\tau_n:=\inf\{t\geq 0: \int_{0}^{t}\left(\tilde l'(L_s^q)L_s^q\right)^2 |q_s|^2\dif s\geq n\}\wedge T.
$$
Since
$$\int_{0}^{T}\left(\tilde l'(L_s^q)L_s^q\right)^2 |q_s|^2\dif s\leq \sup_{t\in\T}\left(\tilde l'(L_t^q)L_t^q\right)^2
\int_{0}^{T}|q_s|^2\dif s<+\infty\ \ \ps,\vspace{0.1cm}$$
it follows that $\tau_n\To T$ as $n$ tends to infinity. We deduce from \eqref{eq:4.20} that for $n\geq 1$ and $t\in \T$,
$$
\begin{array}{l}
\E[\tilde l(L^q_{t\wedge \tau_n})]-\tilde k\exp\left(\mu (\ln \tilde k)^{\frac{1}{1+2\lambda}}\right)\vspace{0.1cm}\\
\ \ \Dis \leq \frac{2\gamma^2\mu^{2\lambda+1}}{1+2\lambda}\E\left[\int_0^{t\wedge \tau_n}\tilde l(L^q_{s}) \dif s\right] +\frac{\mu}{1+2\lambda}\E\left[\int_0^{t\wedge \tau_n} L_s^q |q_s|^2\exp\left(\gamma^{-\frac{1}{\lambda}}|q_s|^{\frac{1}{\lambda}}\right)\dif s\right]+\frac{\mu\bar k_{\lambda,\gamma} T}{1+2\lambda}\vspace{0.1cm}\\
\ \ \Dis \leq  \frac{2\gamma^2\mu^{2\lambda+1}}{1+2\lambda}\int_0^t \E\left[\tilde l(L^q_{s\wedge \tau_n})\right]\dif s +\frac{\mu}{1+2\lambda}\E\left[\int_0^T L_s^q |q_s|^2\exp\left(\gamma^{-\frac{1}{\lambda}}|q_s|^{\frac{1}{\lambda}}\right)\dif s\right]+\frac{\mu\bar k_{\lambda,\gamma} T}{1+2\lambda}.
\end{array}
$$
Applying Gronwall's inequality to the last inequality, we have  for each $n\geq 1$,
$$
\begin{array}{l}
\Dis \E[\tilde l(L^q_{t\wedge \tau_n})]\leq \left(\frac{\mu}{1+2\lambda}\E\left[\int_0^T\!\!\! L_s^q |q_s|^2\exp\left(\gamma^{-\frac{1}{\lambda}}|q_s|^{\frac{1}{\lambda}}\right)\dif s\right]+\frac{\mu\bar k_{\lambda,\gamma} T}{1+2\lambda}+\tilde k\exp\left(\mu (\ln \tilde k)^{\frac{1}{1+2\lambda}}\right) \right)\vspace{0.1cm}\\
\hspace*{2.2cm}\Dis \times \exp\left(\frac{2\gamma^2\mu^{2\lambda+1}}{1+2\lambda}t\right),\quad \ t\in \T.
\end{array}
$$
Letting $t=T$ and sending $n$ to infinity in the last inequality, and then using Fatou's lemma, we obtain that there exists a positive constant $C_{\mu, \gamma, \lambda, T}>0$ depending only on $(\mu, \gamma, \lambda, T)$ such that
\begin{equation}\label{eq:4.21}
\E\left[L^q_{T}\exp\left(\mu \left[\ln(1+L^q_{T})\right]^{\frac{1}{1+2\lambda}}\right)\right]
\leq \Dis C_{\mu, \gamma, \lambda, T}\E\left[\int_0^T L_s^q |q_s|^2\exp\left(\gamma^{-\frac{1}{\lambda}}|q_s|^{\frac{1}{\lambda}}\right)\dif s\right]+ C_{\mu, \gamma, \lambda, T}.
\end{equation}
Since for $\eps>0$,
$$
\lim\limits_{x\To +\infty} \frac{\eps\exp\left(\gamma^{-\frac{1}{\lambda}}|x|^{\frac{1}{\lambda}}\right)}{C_{\mu, \gamma, \lambda, T}x^2}=+\infty,\vspace{0.1cm}
$$
there exists a positive constant $\bar C_{\eps, \mu, \gamma, \lambda, T}>0$ depending only on $(\eps, \mu, \gamma, \lambda, T)$ such that for $x\in \R^+$,
\begin{equation}\label{eq:4.22}
C_{\mu, \gamma, \lambda, T}\, x^2\exp\left(\gamma^{-\frac{1}{\lambda}}|x|^{\frac{1}{\lambda}}\right)\leq \eps \exp\left(2\gamma^{-\frac{1}{\lambda}}|x|^{\frac{1}{\lambda}}\right)+\bar C_{\eps, \mu, \gamma, \lambda, T}.
\end{equation}
The inequality \eqref{eq:4.14} follows immediately by combining \eqref{eq:4.21} and \eqref{eq:4.22}. The proof is complete.
\end{proof}

\subsection{Proof of Assertion (i) of \cref{thm:3.1}\vspace{0.2cm}}

Assume that the core function $f$ satisfies (A0) and (A1), and that
\begin{equation}\label{eq:4.26}
|\xi| \in \bigcap_{\mu>0} \exp(\mu L).
\end{equation}
First of all, it follows from (A0), (A1) and \cref{rmk:3.0} that $\as$, $f(t,\bar q_t)=0$ and $\bar q_t=0$. Then, by \eqref{eq:4.26} we have
\begin{equation}\label{eq:4.27}
\E_{\Q^{\bar q}}\left[|\xi|+\int_0^T f(s,\bar q_s)\dif s\right] =\Dis \E\left[|\xi|\right]<+\infty,\vspace{0.1cm}
\end{equation}
which means that $(\bar q_t)_{t\in\T}\in \qcal(\xi,f)$, and then the space $\qcal(\xi,f)$ is nonempty. \vspace{0.2cm}

Furthermore, we deduce from (A0), (A1) and \cref{rmk:3.0} that $\as$, for each $z\in \R^d$,
$$
0\leq \, g(\omega,t,z)=\sup_{q\in \R^d} \left\{z\cdot q-f(\omega,t,q)\right\}\leq \sup_{q\in \R^d} \left\{z\cdot q-\frac{1}{2\gamma} |q|^2 \right\}
\leq \frac{\gamma}{2}|z|^2,
$$
which yields that the generator $g$ defined in \eqref{eq:3.4} satisfies assumptions (H0) and (H1) with $h_\cdot= 0$. Then we see from (i) of \cref{pro:2.2} and \eqref{eq:4.26} that BSDE \eqref{eq:3.3} admits a unique adapted solution $(Y_t,Z_t)_{t\in\T}$ such that
\begin{equation}\label{eq:4.31}
\sup_{t\in \T}|Y_t|\in \bigcap\limits_{\mu>0}\exp(\mu L).\vspace{0.2cm}
\end{equation}

To show the dual representation, we need to further verify that $U_t(\xi)=Y_t$ for each $t\in \T$, where $U_t(\xi)$ is defined in \eqref{eq:3.2}. We first prove that for each $(q_t)_{t\in\T}\in \qcal(\xi,f)$, it holds that
\begin{equation}\label{eq:4.32}
\E_{\Q^q}\left[\left.\xi+\int_t^Tf(s,q_s)\dif s\right|\F_t\right]\geq Y_t,\ \ t\in\T.
\end{equation}
According to (A0) and \eqref{eq:3.4}, by the dual representation of a convex function, we know that
\begin{equation}\label{eq:4.33}
f(\omega,t,q)=\sup_{z\in \R^d}(z\cdot q -g(\omega,t,z)),\ \ \RE (\omega,t,q)\in \Omega\tim \T\tim \R^d.
\end{equation}
For each $n\geq 1$ and $t\in \T$, define the following stopping time:
\begin{equation}\label{eq:4.34}
\tau_n^t:=\inf\{s\geq t: \int_{t}^{s} |Z_u|^2\dif u\geq n\}\wedge T.
\end{equation}
It follows from \eqref{eq:3.3} and \eqref{eq:4.33} that for each $n\geq 1$,
\begin{equation}\label{eq:4.34*}
\begin{array}{lll}
Y_t&=& \Dis Y_{\tau_n^t}-\int_t^{\tau_n^t} g(s,Z_s)\dif s+\int_t^{\tau_n^t} Z_s\cdot \dif B_s\\
&=& \Dis Y_{\tau_n^t}+\int_t^{\tau_n^t}(Z_s\cdot q_s-g(s,Z_s))\dif s+\int_t^{\tau_n^t} Z_s\cdot \dif B^q_s\\
&\leq & \Dis Y_{\tau_n^t}+\int_t^{\tau_n^t}f(s,q_s)\dif s+\int_t^{\tau_n^t} Z_s\cdot \dif B^q_s,\ \ t\in \T,
\end{array}
\end{equation}
where the shifted Brownian motion $B^q_t:=B_t-\int_0^t q_s\dif s,\ t\in\T$ is a standard $d$-dimensional Brownian motion under the new probability measure $\Q^q$ in light of Girsanov's theorem. Taking the mathematical expectation conditioned on $\F_t$ with respect to $\Q^q$ in the last inequality, we obtain that for each $n\geq 1$,
\begin{equation}\label{eq:4.35}
Y_t\leq \E_{\Q^q}\left[\left.Y_{\tau_n^t}+\int_t^{\tau_n^t}f(s,q_s)\dif s\right|\F_t\right], \ \ t\in \T.
\end{equation}
On the other hand, by (A1) and \eqref{eq:3.1} we have
$$
\frac{1}{2\gamma}\E\left[\int_0^T L^q_t |q_t|^2\dif t\right]
=\frac{1}{2\gamma}\E_{\Q^q}\left[\int_0^T |q_t|^2\dif t\right]\leq \E_{\Q^q}\left[\int_0^T f(t,q_t)\dif t\right]<+\infty,
$$
which together with \eqref{eq:4.12} yields that $\E[L^q_{T}\ln(1+L^q_{T})]<+\infty$, and then, in light of \eqref{eq:4.6} with $\mu=1$ and \eqref{eq:4.31}, we have
$$
\E_{\Q^q}\left[\sup_{t\in \T}|Y_t|\right]=\E\left[\left(\sup_{t\in \T}|Y_t|\right)L^q_{T}\right]\leq
\E\left[\exp\left(\sup_{t\in \T}|Y_t|\right)\right]+\E\left[L^q_{T}\ln(1+L^q_{T})\right]<+\infty.
$$
Thus, the desired assertion \eqref{eq:4.32} follows by sending $n$ to infinity in \eqref{eq:4.35} and applying Lebesgue's dominated convergence theorem.\vspace{0.2cm}

Next, we set $\tilde q_s\in \partial g(s,Z_s)$ for each $s\in \T$ and prove that $(\tilde q_t)_{t\in\T}\in \qcal(\xi,f)$ and
\begin{equation}\label{eq:4.36}
Y_t=\E_{\Q^{\tilde q}}\left[\left.\xi+\int_t^T f(s,\tilde q_s)\dif s\right|\F_t\right],\ \ t\in\T.
\end{equation}
Since $f(s,\tilde q_s)=Z_s\cdot \tilde q_s-g(s,Z_s)$, by (A1) we have for each $s\in \T$,
$$
g(s,Z_s)\leq Z_s\cdot \tilde q_s-\frac{1}{2\gamma}|\tilde q_s|^2\leq \gamma |Z_s|^2+\frac{1}{4\gamma}|\tilde q_s|^2-\frac{1}{2\gamma}|\tilde q_s|^2=\gamma |Z_s|^2-\frac{1}{4\gamma}|\tilde q_s|^2,
$$
and then
$$
\int_0^T |\tilde q_s|^2\dif s\leq 4\gamma \int_0^T \left(\gamma |Z_s|^2-g(s,Z_s)\right)\dif s<+\infty.
$$
Moreover, for each $n\geq 1$, define the following stopping time:
\begin{equation}\label{eq:4.37}
\sigma_n:=\inf\{t\geq 0: \int_{0}^{t}\left(|Z_s|^2+|\tilde q_s|^2\right)\dif s\geq n\}\wedge T.
\end{equation}
and define
$$
L_t^{\tilde q}:=\exp\Big(\int_{0}^{t} \tilde q_s\cdot \dif B_s-\frac{1}{2}\int_{0}^{t}|\tilde q_s|^2\dif s\Big), \ t\in[0,T]\ \ {\rm and}\ \ \frac{\dif \Q^{\tilde q}_n}{\dif \p}:=L^{\tilde q}_{\sigma_n}.\vspace{0.1cm}
$$
Then, $\Q^{\tilde q}_n$ is a probability measure on $(\Omega,\F_T)$ equivalent to $\p$ for each $n\geq 1$. Letting $q_s=\tilde q_s{\bf 1}_{[0,\sigma_n]}(s)$ in \eqref{eq:4.12} of \cref{pro:4.3}, we obtain that for each $n\geq 1$,
\begin{equation}\label{eq:4.38}
\E\left[L^{\tilde q}_{\sigma_n}\ln(1+L^{\tilde q}_{\sigma_n})\right]\leq \frac{1}{2}\E\left[\int_0^{\sigma_n} L^{\tilde q}_t|\tilde q_t|^2\dif t\right]+\ln 2.
\end{equation}
Note that $f(s,\tilde q_s)=Z_s\cdot \tilde q_s-g(s,Z_s)$. By using a similar argument to \eqref{eq:4.34*} and \eqref{eq:4.35}, we can obtain that for each $n\geq 1$,
\begin{equation}\label{eq:4.39}
Y_0=\E_{\Q^{\tilde q}_n}\left[Y_{\sigma_n}+\int_0^{\sigma_n}f(s,\tilde q_s)\dif s\right].\vspace{0.1cm}
\end{equation}
On the other hand, by (A1), \eqref{eq:4.6} with $\mu=\frac{1}{2\gamma}$ and \eqref{eq:4.38}, we deduce that
\begin{equation}\label{eq:4.40}
\begin{array}{l}
\Dis \E_{\Q^{\tilde q}_n}\left[Y_{\sigma_n}+\int_0^{\sigma_n} f(s,\tilde q_s)\dif s\right]\geq \Dis  -\E\left[|Y_{\sigma_n}| L^{\tilde q}_{\sigma_n}\right]+\frac{1}{2\gamma}
\E\left[\left(\int_0^{\sigma_n} |\tilde q_s|^2\dif s\right)L^{\tilde q}_{\sigma_n}\right]\vspace{0.2cm}\\
\ \ \ \geq \Dis  \Dis -\frac{1}{2\gamma}\E\big[\exp\left(2\gamma |Y_{\sigma_n}|\right)\big]-\frac{1}{2\gamma}\E\left[ L^{\tilde q}_{\sigma_n}\ln\left(1+L^{\tilde q}_{\sigma_n}\right)
\right]+\frac{1}{2\gamma}\E\left[\int_0^{\sigma_n} L^{\tilde q}_s |\tilde q_s|^2\dif s\right]\vspace{0.2cm}\\
\ \ \ \geq \Dis -\frac{1}{2\gamma}\E\left[\exp\left(2\gamma \sup\limits_{t\in\T}|Y_t|\right)\right]
+\frac{1}{4\gamma}\E\left[\int_0^{\sigma_n} L^{\tilde q}_s |\tilde q_s|^2\dif s\right]-\frac{\ln 2}{2\gamma}.
\end{array}
\end{equation}
In light of \eqref{eq:4.31}, \eqref{eq:4.39} and \eqref{eq:4.40}, there exists a positive constant $C>0$ independent of $n$ such that $\sup_{n\geq 1}\E[\int_0^{\sigma_n} L^{\tilde q}_s |\tilde q_s|^2\dif s]\leq C$ and then, in light of \eqref{eq:4.38},
\begin{equation}\label{eq:4.41}
\sup_{n\geq 1}\E\left[L^{\tilde q}_{\sigma_n}\ln(1+L^{\tilde q}_{\sigma_n})\right]\leq \frac{C}{2}+\ln 2<+\infty.
\end{equation}
According to the de La Vall\'{e}e-Poussin lemma and the last inequality, we deduce that the random variable sequence $(L^{\tilde q}_{\sigma_n})_{n=1}^{\infty}$ is uniformly integrable, and then $\E[L^{\tilde q}_T]=1$, so $(L^{\tilde q}_t)_{t\in\T}$ is a uniformly integrable martingale. Furthermore, by applying Fatou's lemma and \eqref{eq:4.41}, we obtain
\begin{equation}\label{eq:4.42}
\E\left[L^{\tilde q}_T\ln(1+L^{\tilde q}_T)\right]\leq \liminf_{n\To \infty} \E\left[L^{\tilde q}_{\sigma_n}\ln(1+L^{\tilde q}_{\sigma_n})\right]<+\infty.
\end{equation}
To show that $(\tilde q_t)_{t\in\T}\in \qcal(\xi,f)$, we have to prove that
\begin{equation}\label{eq:4.43}
\E_{\Q^{\tilde q}}\left[|\xi|+\int_0^T f(s,\tilde q_s)\, \dif s\right]<+\infty.
\end{equation}
Indeed, by \eqref{eq:4.6} with $\mu=1$, \eqref{eq:4.26} and \eqref{eq:4.42} we have
\begin{equation}\label{eq:4.44}
\Dis \E_{\Q^{\tilde q}}\left[|\xi|\right]
\leq \Dis \E\left[\exp(|\xi|)\right]+\E\left[L^{\tilde q}_T\ln(1+L^{\tilde q}_T)\right]<+\infty.
\end{equation}
By a similar computation to \eqref{eq:4.34*} and \eqref{eq:4.35}, we  deduce that for each $n\geq 1$,
\begin{equation}\label{eq:4.45}
Y_t=\E_{\Q^{\tilde q}}\left[\left. Y_{\tau^t_n}+\int_t^{\tau^t_n} f(s,\tilde q_s)\, \dif s\right|\F_t\right], \ \ t\in \T.
\end{equation}
Then, in light of \eqref{eq:4.45} with $t=0$, \eqref{eq:4.6} with $\mu=1$, \eqref{eq:4.31} and \eqref{eq:4.42}, we have
$$
\sup_{n\geq 1}\E_{\Q^{\tilde q}}\left[\int_0^{\tau^0_n} f(s,\tilde q_s)\, \dif s\right]\leq  |Y_0|+\E_{\Q^{\tilde q}}\left[\sup_{t\in\T}|Y_t|\right]<+\infty,
$$
which together with Fatou's lemma and \eqref{eq:4.44} yields \eqref{eq:4.43}. Sending $n$ to infinity in \eqref{eq:4.45} and applying Lebesgue's dominated convergence theorem,  we get \eqref{eq:4.36}.\vspace{0.2cm}

Finally, note that $U_t(\xi)=Y_t$ for each $t\in \T$. The existence and uniqueness of the adapted solution of BSDE \eqref{eq:3.3} satisfying the prescribed integrability along with the related comparison theorem, yield that the operator process $\{U_t(\cdot),\ t\in\T\}$ defined via \eqref{eq:3.2} satisfies (i)-(v) in the introduction and is therefore a dynamic concave utility defined on $\bigcap_{\mu>0} \exp(\mu L)$.\vspace{0.2cm}

In fact, for a constant $c$, $(Y_t,Z_t)_{t\in\T}:= (c,0)$ is obviously  the unique adapted solution of BSDE \eqref{eq:3.3} with $\xi= c$ such that
$\sup_{t\in\T}|Y_t|\, \in \, \bigcap_{\mu>0}\exp(\mu L)$, and thus $U_t(c)\equiv c$. Consequently, we have the preservation  of constants.\vspace{0.2cm}

Second, for $i=1,2$, let $\xi^i \in \bigcap_{\mu>0}\exp(\mu L)$ and $(Y_s^i,Z_s^i)_{s\in\T}$ be the unique solution of BSDE \eqref{eq:3.3} with $\xi=\xi^i$ such that
$\sup_{s\in\T}|Y_s^i|\in \bigcap_{\mu>0}\exp(\mu L)$. From \citet[Theorem 5]{BriandHu2008PTRF}, we see that $Y_t^1\geq Y_t^2$ and thus $U_t(\xi^1)\geq U_t(\xi^2)$ for $t\in\T$, if $\xi^1\geq \xi^2$. Therefore, we have the monotonicity.\vspace{0.2cm}

Third, for  $\bar t\in\T, \xi\in \bigcap_{\mu>0}\exp(\mu L)$ and $\eta\in \bigcap_{\mu>0}\exp(\mu L)(\F_{\bar t})$, let $(Y_s,Z_s)_{s\in\T}$ be the unique solution of BSDE \eqref{eq:3.3} such that
$\sup_{s\in\T}|Y_s|\in \bigcap_{\mu>0}\exp(\mu L)$, and $(\bar Y_s,\bar Z_s)_{s\in [0,\bar t\,]}$ the unique solution of BSDE \eqref{eq:3.3} with $(T,\xi)$ being replaced with $(\bar t,\eta)$ such that $\sup_{s\in [0,\bar t\,]}|\bar Y_s|\in \bigcap_{\mu>0}\exp(\mu L)(\F_{\bar t})$. It is not difficult to verify that
$$
\left(\tilde Y_s,\ \tilde Z_s\right):=\left(\bar Y_s{\bf 1}_{s\in [0,\bar t\,)}+(Y_s+\eta){\bf 1}_{s\in [\bar t,T]},\  \bar Z_s{\bf 1}_{s\in [0,\bar t\,)}+Z_s{\bf 1}_{s\in [\bar t,T]}\right), \quad s\in\T
$$
is the unique solution of BSDE \eqref{eq:3.3} with $\xi$ being replaced with $\xi+\eta$ such that $\sup_{s\in [0,T]}|\tilde Y_s|\in \bigcap_{\mu>0}\exp(\mu L)$, and therefore,  $U_{\bar t}(\xi+\eta)=\tilde Y_{\bar t}=Y_{\bar t}+\eta=U_{\bar t}(\xi)+\eta$. Then, we have the translation invariance.\vspace{0.2cm}

Fourth, for the notations introduced in the second last paragraph, we have  $U_t(\xi^i)=Y_t^i$  for $i=1,2$ and $t\in\T$.   Furthermore, for $\theta\in [0,1]$, the pair
$$(Y_t, Z_t):=\left(\theta Y_t^1+(1-\theta)Y_t^2, \, \theta Z_t^1+(1-\theta) Z_t^2\right), \quad t\in \T$$ solves the following BSDE
\begin{equation}\label{eq:4.40-1}
Y_t=\theta\xi^1+(1-\theta)\xi^2-\int_t^T g(s) \, {\rm d}s+\int_t^T Z_s \cdot {\rm d}B_s, \quad  t\in \T,
\end{equation}
where $g(t):=\theta g(t,Z_t^1)+(1-\theta)g(t,Z_t^2)$ for $t\in \T$. In what follows, let $(\hat Y_t,\hat Z_t)_{t\in\T}$ be the unique solution of BSDE \eqref{eq:3.3} with $\xi$ being replaced with $\theta\xi^1+(1-\theta)\xi^2$  such that $\sup_{s\in\T}|\hat Y_s|\in \bigcap_{\mu>0}\exp(\mu L)$, that is
\begin{equation}\label{eq:4.40-2}
\hat Y_t=\theta\xi^1+(1-\theta)\xi^2-\int_t^T g(s,\hat Z_s)\,  {\rm d}s+\int_t^T \hat Z_s \cdot {\rm d}B_s, \quad  t\in \T.
\end{equation}
Then, $U_t(\theta\xi^1+(1-\theta)\xi^2)=\hat Y_t$  for each $t\in \T$. Since $g$  satisfies assumptions (H0) and (H1), we have
$$
-g(t,Z_t)=-g(t,\theta Z_t^1+(1-\theta) Z_t^2)\geq -\theta g(t,Z_t^1)-(1-\theta)g(t, Z_t^2)=-g(t), \quad  t\in \T.
$$
In light of BSDEs \eqref{eq:4.40-1} and \eqref{eq:4.40-2}, we have from
\citet[Theorem 5]{BriandHu2008PTRF}  that $\hat Y_t\geq Y_t$ for $t\in \T$, and thus
$$U_t(\theta\xi^1+(1-\theta)\xi^2)\geq \theta U_t(\xi^1)+(1-\theta)U_t(\xi^2).$$
Then, we have the concavity.\vspace{0.2cm}

At last, for each $\bar t\in\T$ and $\xi\in \bigcap_{\mu>0}\exp(\mu L)$, let $(Y_s,Z_s)_{s\in\T}$ be the unique solution of BSDE \eqref{eq:3.3} such that
$\sup_{s\in\T}|Y_s|\in \bigcap_{\mu>0}\exp(\mu L)$, and $(\bar Y_s,\bar Z_s)_{s\in [0,T]}$ the unique solution of BSDE \eqref{eq:3.3} with $\xi$ being replaced with $Y_{\bar t}$ such that $\sup_{s\in [0,T]}|\bar Y_s|\in \bigcap_{\mu>0}\exp(\mu L)$. Then,  we have
$$U_s(\xi)=Y_s\ \ {\rm and}\ \ U_s(U_{\bar t}(\xi))=U_s(Y_{\bar t})=\bar Y_s, \quad  s\in [0,\bar t\,].$$
Furthermore, let $(\tilde Y_s,\tilde Z_s)_{s\in [0,\bar t\,]}$ be the unique solution of BSDE \eqref{eq:3.3} with $(T,\xi)$ being replaced with $(\bar t,Y_{\bar t})$ such that $\sup_{s\in [0,\bar t\,]}|\tilde Y_s|\in \bigcap_{\mu>0}\exp(\mu L)(\F_{\bar t})$. Since $g(t,0)=0$ due to (A0) and \eqref{eq:3.4}, we see from  (i) of \cref{pro:2.2}  that
$$
(Y_s, Z_s)=(\tilde Y_s{\bf 1}_{s\in [0,\bar t\,]}+Y_s{\bf 1}_{s\in (\bar t,T]}, \, \tilde Z_s{\bf 1}_{s\in [0,\bar t\,]}+Z_s{\bf 1}_{s\in (\bar t,T]}), \quad s\in\T\vspace{-0.1cm}
$$
and
$$
(\bar Y_s, \bar Z_s)=(\tilde Y_s{\bf 1}_{s\in [0,\bar t\,]}+Y_{\bar t}{\bf 1}_{s\in (\bar t,T]}, \, \tilde Z_s{\bf 1}_{s\in [0,\bar t\,]}), \quad s\in\T.\vspace{0.1cm}
$$
Hence,  $Y_s=\tilde Y_s=\bar Y_s$ for $s\in [0,\bar t\,]$. That is to say, $U_s(\xi)=U_s(U_{\bar t}(\xi))$ for each $0\leq s\leq \bar t\leq T$. Then, we have the time consistency. The proof of Assertion (i) of \cref{thm:3.1} is complete.\vspace{0.2cm}

\subsection{Proof of Assertion (ii) of \cref{thm:3.1}\vspace{0.2cm}}

Assume that the core function $f$ satisfies (A0) and (A2), and that
\begin{equation}\label{eq:4.46}
|\xi|\, \in \, \bigcap_{\mu>0} \exp(\mu L^{\frac{2}{\alpha^*}}).
\end{equation}
First of all, it follows from (A0), (A2) and \cref{rmk:3.0} that $\as$, $f(t,\bar q_t)=0$ and $\bar q_t=0$. Moreover, by \eqref{eq:4.27} and \eqref{eq:4.46} we have $(\bar q_t)_{t\in\T}\in \qcal(\xi,f)$, and then the space $\qcal(\xi,f)$ is nonempty.\vspace{0.2cm}

Furthermore, by virtue of (A0), (A2), \cref{rmk:3.0}, \eqref{eq:3.4} and \eqref{eq:4.8} with $\mu=\gamma$ and $\delta=\alpha$ we can deduce that $\as$, for each $z\in \R^d$,
$$
0\leq g(\omega,t,z)=\sup_{q\in \R^d} \left\{z\cdot q-f(\omega,t,q)\right\}\leq \sup_{q\in \R^d} \left\{z\cdot q-{\gamma^{-\frac{1}{\alpha-1}}} |q|^{\alpha^*}\right\}
\leq \gamma |z|^{\alpha},
$$
which yields that the generator $g$ defined in \eqref{eq:3.4} satisfies assumptions (H0) and (H2) with $h_\cdot= 0$. It then follows from (ii) of \cref{pro:2.2} and \eqref{eq:4.46} that BSDE \eqref{eq:3.3} admits a unique adapted solution $(Y_t,Z_t)_{t\in\T}$ such that
\begin{equation}\label{eq:4.51}
\sup_{t\in \T}|Y_t|\, \in \, \bigcap\limits_{\mu>0}\exp(\mu L^{\frac{2}{\alpha^*}}).
\end{equation}

To show the dual representation, we need to further verify that $U_t(\xi)=Y_t$ for each $t\in \T$, where $U_t(\xi)$ is defined in \eqref{eq:3.2}. We first prove that for each $(q_t)_{t\in\T}\in \qcal(\xi,f)$, it holds that
\begin{equation}\label{eq:4.52}
\E_{\Q^q}\left[\left.\xi+\int_t^Tf(s,q_s)\, \dif s\right|\F_t\right]\geq Y_t,\ \ t\in\T.
\end{equation}
In light of (A0), \eqref{eq:3.3} and \eqref{eq:3.4}, a same computation to those among \eqref{eq:4.33}-\eqref{eq:4.35} yields that for each $n\geq 1$,
\begin{equation}\label{eq:4.53}
Y_t\leq \E_{\Q^q}\left[\left.Y_{\tau_n^t}+\int_t^{\tau_n^t}f(s,q_s)\, \dif s\right|\F_t\right], \ \ t\in \T,
\end{equation}
where the stopping time $\tau_n^t$ is defined in \eqref{eq:4.34}. On the other hand, by (A2) and \eqref{eq:3.1} we have
$$
{\gamma^{-\frac{1}{\alpha-1}}}\, \E\left[\int_0^T L^q_t |q_t|^{\alpha^*}\dif t\right]= {\gamma^{-\frac{1}{\alpha-1}}}\, \E_{\Q^q}\left[\int_0^T |q_t|^{\alpha^*}\dif t\right]\leq  \E_{\Q^q}\left[\int_0^T f(t,q_t)\dif t\right]<+\infty,
$$
which together with \eqref{eq:4.13} yields that $\E\left[L^q_{T}[\ln(1+L^q_{T})\right]^{\frac{\alpha^*}{2}}]<+\infty$, and then, in light of \eqref{eq:4.51} and \eqref{eq:4.4} with $(\eps,q,\mu,\delta)=(1,2,1,\frac{2}{\alpha^*})$, we have
$$
\begin{array}{l}
\Dis\E_{\Q^q}\left[\sup_{t\in \T}|Y_t|\right]= \Dis \E\left[\left(\sup_{t\in \T}|Y_t|\right)L^q_{T}\right]\vspace{0.1cm}\\
\ \ \leq \Dis
\E\left[\exp\left\{2\left(\sup_{t\in \T}|Y_t|\right)^{\frac{2}{\alpha^*}}\right\}\right]+\E\left[L^q_{T}
\left[\ln(1+L^q_{T})\right]^{\frac{\alpha^*}{2}}\right]+\tilde C_\alpha<+\infty,
\end{array}
$$
where $\tilde C_\alpha$ is a positive constant depending only on $\alpha$. Thus, the desired assertion \eqref{eq:4.52} follows by sending $n$ to infinity in \eqref{eq:4.53} and applying Lebesgue's dominated convergence theorem.\vspace{0.2cm}

Next, we set $\tilde q_s\in \partial g(s,Z_s)$ for each $s\in \T$ and prove that $(\tilde q_t)_{t\in\T}\in \qcal(\xi,f)$ and
\begin{equation}\label{eq:4.54}
Y_t=\E_{\Q^{\tilde q}}\left[\left.\xi+\int_t^T f(s,\tilde q_s)\, \dif s\right|\F_t\right],\ \ t\in\T.
\end{equation}
Since $f(s,\tilde q_s)=Z_s\cdot \tilde q_s-g(s,Z_s)$, by (A2) and \eqref{eq:4.8} with $\mu=2^{\alpha-1}\gamma$ and $\delta=\alpha$ we have for each $s\in \T$,
$$
g(s,Z_s)\leq Z_s\cdot \tilde q_s-{\gamma^{-\frac{1}{\alpha-1}}}|\tilde q_s|^{\alpha^*}\leq 2^{\alpha-1}\gamma |Z_s|^\alpha+\frac{1}{2}{\gamma^{-\frac{1}{\alpha-1}}}|\tilde q_s|^{\alpha^*}-{\gamma^{-\frac{1}{\alpha-1}}}|\tilde q_s|^{\alpha^*},
$$
and then, in light of the basic assumption of $1<\alpha<2<\alpha^*$,
$$
\int_0^T |\tilde q_s|^2 \dif s\leq \int_0^T (1+|\tilde q_s|^{\alpha^*})\, \dif s\leq T+ 2\gamma^{\frac{1}{\alpha-1}}\int_0^T \left[2^{\alpha-1}\gamma (1+|Z_s|^2)-g(s,Z_s)\right]\dif s<+\infty.
$$
Moreover, by letting $q_s=\tilde q_s{\bf 1}_{[0,\sigma_n]}(s)$ in \eqref{eq:4.13} of \cref{pro:4.3}, we can conclude that there exists a positive constant $C_{\alpha,T}$ depending only on $(\alpha,T)$ such that for each $n\geq 1$,
\begin{equation}\label{eq:4.55}
\E\left[L^{\tilde q}_{\sigma_n}\left[\ln(1+L^{\tilde q}_{\sigma_n})\right]^{\frac{\alpha^*}{2}}\right]\leq C_{\alpha,T}\E\left[\int_0^{\sigma_n} L^{\tilde q}_t|\tilde q_t|^{\alpha^*}\dif t\right]+C_{\alpha,T},
\end{equation}
where the stopping time $\sigma_n$ is defined in \eqref{eq:4.37}. In light of \eqref{eq:3.3} together with the fact that $f(s,\tilde q_s)=Z_s\cdot \tilde q_s-g(s,Z_s)$,  by an identical analysis to \eqref{eq:4.39} we obtain that for each $n\geq 1$,
\begin{equation}\label{eq:4.56}
Y_0=\E_{\Q^{\tilde q}_n}\left[Y_{\sigma_n}+\int_0^{\sigma_n}f(s,\tilde q_s)\, \dif s\right].
\end{equation}
On the other hand, applying assumption (A2), inequality \eqref{eq:4.4} with $$(\eps,q,\delta,\mu)=(1,\, 2,\, \frac{2}{\alpha^*},\, \frac{1}{2}C_{\alpha,T}^{-1}{\gamma^{-\frac{1}{\alpha-1}}})$$
and inequality \eqref{eq:4.55}, we obtain that there exists a $\bar C_{\alpha,\gamma,T}>0$ depending only on $(\alpha,\gamma,T)$ such that
\begin{equation}\label{eq:4.57}
\begin{array}{l}
\Dis \E_{\Q^{\tilde q}_n}\left[Y_{\sigma_n}+\int_0^{\sigma_n} f(s,\tilde q_s)\, \dif s\right]\geq \Dis  -\E\left[|Y_{\sigma_n}| L^{\tilde q}_{\sigma_n}\right]+{\gamma^{-\frac{1}{\alpha-1}}}
\E\left[\left(\int_0^{\sigma_n} |\tilde q_s|^{\alpha^*}\dif s\right)L^{\tilde q}_{\sigma_n}\right]\vspace{0.1cm}\\
\ \ \ \geq \Dis  \Dis -\E\left[\exp\left(2(2\gamma^{\frac{1}{\alpha-1}}C_{\alpha,T})^{\frac{2}{\alpha^*}}
|Y_{\sigma_n}|^{\frac{2}{\alpha^*}}\right)\right]-\bar C_{\alpha,\gamma,T}\vspace{0.1cm}\\
\ \ \ \ \Dis \ \ -
\frac{1}{2}\gamma^{-\frac{1}{\alpha-1}}C_{\alpha,T}^{-1}\, \E\left[ L^{\tilde q}_{\sigma_n}\left[\ln\left(1+L^{\tilde q}_{\sigma_n}\right)\right]^{\frac{\alpha^*}{2}}
\right]+{\gamma^{-\frac{1}{\alpha-1}}}\E\left[\int_0^{\sigma_n} L^{\tilde q}_s |q_s|^{\alpha^*}\dif s\right]\vspace{0.1cm}\\
\ \ \ \geq \Dis -\E\left[\exp\left(2(2\gamma^{\frac{1}{\alpha-1}}C_{\alpha,T})^{\frac{2}{\alpha^*}}
\sup\limits_{t\in\T}|Y_t|^{\frac{2}{\alpha^*}}\right)\right]-\bar C_{\alpha,\gamma,T}\vspace{0.1cm}\\
\qquad \Dis+\frac{1}{2}\gamma^{-\frac{1}{\alpha-1}}\, \E\left[\int_0^{\sigma_n} L^{\tilde q}_s |q_s|^{\alpha^*}\dif s\right]-\frac{1}{2}\gamma^{-\frac{1}{\alpha-1}}.
\end{array}
\end{equation}
In light of \eqref{eq:4.51} and \eqref{eq:4.55}-\eqref{eq:4.57}, by an identical analysis to \eqref{eq:4.42} we can conclude that the process $(L^{\tilde q}_t)_{t\in\T}$ is a uniformly integrable martingale, and
\begin{equation}\label{eq:4.59}
\E\Big[L^{\tilde q}_T\left[\ln(1+L^{\tilde q}_T)\right]^{\frac{{\alpha^*}}{2}}\Big]\leq \liminf_{n\To \infty} \E\Big[L^{\tilde q}_{\sigma_n}\left[\ln(1+L^{\tilde q}_{\sigma_n})\right]^{\frac{{\alpha^*}}{2}}\Big]<+\infty.
\end{equation}
To show that $(\tilde q_t)_{t\in\T}\in \qcal(\xi,f)$, we have to prove that
\begin{equation}\label{eq:4.60}
\E_{\Q^{\tilde q}}\left[|\xi|+\int_0^T f(s,\tilde q_s)\, \dif s\right]<+\infty.
\end{equation}
Indeed, by \eqref{eq:4.4} with $(\eps,q,\mu,\delta)=(1,2,1,\frac{2}{\alpha^*})$, \eqref{eq:4.46} and \eqref{eq:4.59}, we have
\begin{equation}\label{eq:4.61}
\E_{\Q^{\tilde q}}\left[|\xi|\right]=\E\left[|\xi|L^{\tilde q}_T\right]\leq \E\left[\exp\left(2 |\xi|^{\frac{2}{\alpha^*}}\right)\right]+\E\Big[L^{\tilde q}_T\left(\ln(1+L^{\tilde q}_T)\right)^{\frac{\alpha^*}{2}}\Big]+\tilde C_\alpha<+\infty,
\end{equation}
where $\tilde C_\alpha$ is a positive constant depending only on $\alpha$. By an identical computation to \eqref{eq:4.45}, we deduce that for each $n\geq 1$,
\begin{equation}\label{eq:4.62}
Y_t=\E_{\Q^{\tilde q}}\left[\left. Y_{\tau^t_n}+\int_t^{\tau^t_n} f(s,\tilde q_s)\, \dif s\right|\F_t\right], \ \ t\in \T.
\end{equation}
Then, in light of \eqref{eq:4.62} with $t=0$, \eqref{eq:4.4} with $(\eps,q,\mu,\delta)=(1,2,1,\frac{2}{\alpha^*})$, \eqref{eq:4.51} and \eqref{eq:4.59}, we have
$$
\sup_{n\geq 1}\E_{\Q^{\tilde q}}\left[\int_0^{\tau^0_n} f(s,\tilde q_s)\, \dif s\right]\leq  |Y_0|+\E_{\Q^{\tilde q}}\left[\sup_{t\in\T}|Y_t|\right]<+\infty,
$$
which together with Fatou's lemma and \eqref{eq:4.61} yields \eqref{eq:4.60}. Sending $n\To \infty$ in \eqref{eq:4.62} and applying Lebesgue's dominated convergence theorem, we get \eqref{eq:4.54}. The rest proof runs as (i) of \cref{thm:3.1}.

\subsection{Proof of Assertion (iii) of \cref{thm:3.1}\vspace{0.2cm}}

Assume that the core function $f$ satisfies (A0) and (A3), and that there exists a positive constant $\bar\mu>0$ such that
\begin{equation}\label{eq:4.63}
|\xi|\in L\exp[\bar\mu (\ln L)^{1+2\lambda}].\vspace{0.1cm}
\end{equation}
First of all, it follows from (A0), (A3) and \cref{rmk:3.0} that $\as$, $f(t,\bar q_t)=0$ and $\bar q_t=0$. Moreover, by \eqref{eq:4.27} and \eqref{eq:4.63} we have $(\bar q_t)_{t\in\T}\in \qcal(\xi,f)$, and then the space $\qcal(\xi,f)$ is nonempty.\vspace{0.2cm}

Furthermore, by (A0), (A3), \cref{rmk:3.0}, \eqref{eq:3.4} and \eqref{eq:4.4} with $(\eps,q,\mu,\delta)=(c,2, \gamma,\frac{1}{\lambda})$ we can deduce that there exists a constant $C_{c,\gamma,\lambda}>0$ depending only on $(c,\gamma,\lambda)$ such that $\as$, $\RE z\in \R^d$,
$$
\begin{array}{lll}
0\leq g(\omega,t,z)&=&\Dis \sup_{q\in \R^d} \left\{z\cdot q-f(\omega,t,q)\right\}\leq \sup_{q\in \R^d} \left\{z\cdot q-c\exp\left(2{\gamma^{-\frac{1}{\lambda}}}
|q|^{\frac{1}{\lambda}}\right)+c\right\}\\
&\leq& \Dis \gamma |z|\left(\ln (1+|z|)\right)^\lambda +C_{c,\gamma,\lambda},
\end{array}
$$
which yields that the generator $g$ defined in \eqref{eq:3.4} satisfies assumptions (H0) and (H3) with $h_\cdot=C_{c,\gamma,\lambda}$. It then follows from (iii) of \cref{pro:2.2} and \eqref{eq:4.63} that BSDE \eqref{eq:3.3} admits a unique adapted solution $(Y_t,Z_t)_{t\in\T}$ such that for some positive constant $\tilde\mu<\bar\mu$,
\begin{equation}\label{eq:4.68}
\sup_{t\in \T}|Y_t|\in L\exp[\tilde\mu(\ln L)^{1+2\lambda}].
\end{equation}

To show the dual representation, we need to further verify that $U_t(\xi)=Y_t$ for each $t\in \T$, where $U_t(\xi)$ is defined in \eqref{eq:3.2}. We first prove that for each $(q_t)_{t\in\T}\in \qcal(\xi,f)$, it holds that
\begin{equation}\label{eq:4.69}
\E_{\Q^q}\left[\left.\xi+\int_t^Tf(s,q_s)\dif s\right|\F_t\right]\geq Y_t,\ \ t\in\T.
\end{equation}
In light of (A0), \eqref{eq:3.3} and \eqref{eq:3.4}, a same computation as that from \eqref{eq:4.33}-\eqref{eq:4.35} yields that for each $n\geq 1$,
\begin{equation}\label{eq:4.70}
Y_t\leq \E_{\Q^q}\left[\left.Y_{\tau_n^t}+\int_t^{\tau_n^t}f(s,q_s)\dif s\right|\F_t\right], \ \ t\in \T,
\end{equation}
where the stopping time $\tau_n^t$ is defined in \eqref{eq:4.34}. On the other hand, by (A3) and \eqref{eq:3.1}, we have
$$
c\E\left[\int_0^T L^q_t \exp\left(2{\gamma^{-\frac{1}{\lambda}}}|q_t|^{\frac{1}{\lambda}}\right)\dif t\right]= \Dis c\E_{\Q^q}\left[\int_0^T \exp\left(2{\gamma^{-\frac{1}{\lambda}}}|q_t|^{\frac{1}{\lambda}}\right)\dif t\right]\leq \E_{\Q^q}\left[\int_0^T f(t,q_t)\dif t\right]+cT<+\infty,
$$
which and \eqref{eq:4.14} with $\mu=\tilde\mu^{-\frac{1}{1+2\lambda}}$ yield that
$$\E\big[L^q_T\exp\{\tilde\mu^{-\frac{1}{1+2\lambda}}
[\ln(1+L^q_T)]^{\frac{1}{1+2\lambda}}\}\big]<+\infty,$$
and then, in light of \eqref{eq:4.68} and \eqref{eq:4.3} with $\mu=\tilde\mu$ and $\delta=\frac{1}{1+2\lambda}$, we have
$$
\begin{array}{l}
\E_{\Q^q}\left[\sup_{t\in \T}|Y_t|\right]=\Dis \E\left[\left(\sup_{t\in \T}|Y_t|\right)L^q_{T}\right]\leq  \Dis \E\left[L^q_T\exp\left\{\tilde\mu^{-\frac{1}{1+2\lambda}}
\left[\ln\left(1+L^q_T\right)\right]^{\frac{1}{1+2\lambda}}
\right\}\right]\vspace{0.1cm}\\
\hspace{3.5cm} \Dis +\E\left[\left(\sup_{t\in \T}|Y_t|\right)\exp\left\{\tilde\mu\left[\ln\left(1+\sup_{t\in \T}|Y_t|\right)\right]^{1+2\lambda}\right\}\right]<+\infty.
\end{array}
$$
Thus, the desired assertion \eqref{eq:4.69} follows by sending $n$ to infinity in \eqref{eq:4.70} and applying Lebesgue's dominated convergence theorem.\vspace{0.2cm}

Next, we set $\tilde q_s\in \partial g(s,Z_s)$ for each $s\in \T$ and prove that $(\tilde q_t)_{t\in\T}\in \qcal(\xi,f)$ and
\begin{equation}\label{eq:4.71}
Y_t=\E_{\Q^{\tilde q}}\left[\left.\xi+\int_t^T f(s,\tilde q_s)\dif s\right|\F_t\right],\ \ t\in\T.
\end{equation}
Since $f(s,\tilde q_s)=Z_s\cdot \tilde q_s-g(s,Z_s)$, by (A3) and \eqref{eq:4.4} with $(\eps,q,\mu,\delta)=(\frac{c}{2},2, \gamma,\frac{1}{\lambda})$ we deduce that there exists a constant $\bar C_{c,\gamma,\lambda}>0$ depending only on $(c,\gamma,\lambda)$ such that for each $s\in \T$,
$$
g(s,Z_s)\leq Z_s\cdot \tilde q_s-c\exp\left(2{\gamma^{-\frac{1}{\lambda}}}
|\tilde q_s|^{\frac{1}{\lambda}}\right)+c\leq \gamma |Z_s|\left(\ln (1+|Z_s|)\right)^\lambda-\frac{c}{2}\exp
\left(2{\gamma^{-\frac{1}{\lambda}}}
|\tilde q_s|^{\frac{1}{\lambda}}\right)+\bar C_{c,\gamma,\lambda},
$$
and then, in light of the fact that $x^2\leq \frac{c}{2}\exp(2{\gamma^{-\frac{1}{\lambda}}}
x^{\frac{1}{\lambda}})+\tilde C_{c,\gamma,\lambda}$ for each $x\in\R_+$ and some constant $\tilde C_{c,\gamma,\lambda}>0$ depending only on $(c,\gamma,\lambda)$, and the fact that $x(\ln(1+x))^\lambda\leq K_\lambda+x^2$ for each $x\in\R_+$ and some constant $K_\lambda>0$ depending only on $\lambda$,
$$
\begin{array}{lll}
\Dis \int_0^T |\tilde q_s|^2 \dif s &\leq & \Dis \frac{c}{2}\int_0^T\exp
\left(2{\gamma^{-\frac{1}{\lambda}}}
|\tilde q_s|^{\frac{1}{\lambda}}\right)\dif s+T\tilde C_{c,\gamma,\lambda}\vspace{0.1cm}\\
&\leq & \Dis \int_0^T \left[\gamma (K_\lambda+|Z_s|^2)-g(s,Z_s)\right]\dif s+T\tilde C_{c,\gamma,\lambda} <+\infty.
\end{array}
$$
Moreover, letting $q_s=\tilde q_s{\bf 1}_{[0,\sigma_n]}(s)$ in \eqref{eq:4.14} of \cref{pro:4.3}, we conclude that for each $\eps>0$, there exists a positive constant $C_{\eps,\bar\mu,\gamma,\lambda,T}$ depending only on $(\eps,\bar\mu,\gamma,\lambda,T)$ such that for each $n\geq 1$,
\begin{equation}\label{eq:4.72}
\E\left[L^{\tilde q}_{\sigma_n}\exp\left(
{\bar\mu^{-\frac{1}{1+2\lambda}}}\left[\ln(1+L^{\tilde q}_{\sigma_n})\right]^{\frac{1}{1+2\lambda}}\right)\right]
\leq \Dis \eps\, \E\left[\int_0^{\sigma_n} L_s^{\tilde q} \exp\left(2{\gamma^{-\frac{1}{\lambda}}}
|{\tilde q}_s|^{\frac{1}{\lambda}}\right)\dif s\right]+ C_{\eps,\bar\mu,\gamma,\lambda,T},
\end{equation}
where the stopping time $\sigma_n$ is defined in \eqref{eq:4.37}. In light of \eqref{eq:3.3} together with the fact that $f(s,\tilde q_s)=Z_s\cdot \tilde q_s-g(s,Z_s)$,  by an identical analysis to \eqref{eq:4.39} we obtain that for each $n\geq 1$,
\begin{equation}\label{eq:4.73}
Y_0=\E_{\Q^{\tilde q}_n}\left[Y_{\sigma_n}+\int_0^{\sigma_n}f(s,\tilde q_s)\dif s\right].
\end{equation}
On the other hand, by applying assumption (A3), inequality \eqref{eq:4.3} with $\mu=\bar\mu$ and $\delta=\frac{1}{1+2\lambda}$, and inequality \eqref{eq:4.72} with $\eps=\frac{c}{2}$, we can conclude that
\begin{equation}\label{eq:4.74}
\begin{array}{l}
\Dis \E_{\Q^{\tilde q}_n}\left[Y_{\sigma_n}+\int_0^{\sigma_n} f(s,\tilde q_s)\dif s\right]\vspace{0.1cm}\\
\ \ \geq \Dis  -\E\left[|Y_{\sigma_n}| L^{\tilde q}_{\sigma_n}\right]+c\, \E\left[\left(\int_0^{\sigma_n} \exp\left(2{\gamma^{-\frac{1}{\lambda}}}
|\tilde q_s|^{\frac{1}{\lambda}}\right)\dif s\right)L^{\tilde q}_{\sigma_n}\right]-cT\vspace{0.1cm}\\
\ \ \geq \Dis -\E\left[|Y_{\sigma_n}|\exp\left\{\bar\mu\left[\ln \left(1+|Y_{\sigma_n}|\right)\right]^{1+2\lambda}\right\}\right]
-\E\left[L^{\tilde q}_{\sigma_n}\exp\left\{\left[{\bar\mu^{-\frac{1}{1+2\lambda}}}
\ln\left(1+L^{\tilde q}_{\sigma_n}\right)
\right]^{\frac{1}{1+2\lambda}}\right\}
\right]\vspace{0.1cm}\\
\ \ \ \ \ \Dis +c\E\left[\int_0^{\sigma_n} L^{\tilde q}_s \exp\left(2{\gamma^{-\frac{1}{\lambda}}}
|\tilde q_s|^{\frac{1}{\lambda}}\right)\dif s\right]-cT\vspace{0.1cm}\\
\ \ \geq \Dis -\E\left[\left(\sup\limits_{t\in\T}|Y_t|\right) \exp\left\{\bar\mu\left[\ln \left(1+\sup\limits_{t\in\T}|Y_t|\right)\right]^{1+2\lambda}\right\}\right]- C_{\frac{c}{2},\bar\mu,\gamma,\lambda,T}\vspace{0.1cm}\\
\ \ \ \ \ \Dis +\frac{c}{2}\E\left[\int_0^{\sigma_n} L^{\tilde q}_s \exp\left(2{\gamma^{-\frac{1}{\lambda}}}
|\tilde q_s|^{\frac{1}{\lambda}}\right)\dif s\right]-cT.\vspace{0.1cm}
\end{array}
\end{equation}
In light of \eqref{eq:4.68} and \eqref{eq:4.72}-\eqref{eq:4.74}, by an identical analysis to \eqref{eq:4.42} we can conclude that the process $(L^{\tilde q}_t)_{t\in\T}$ is a uniformly integrable martingale, and
\begin{equation}\label{eq:4.76}
\E\left[L^{\tilde q}_T\exp\left({\bar\mu^{-\frac{1}{1+2\lambda}}}\left[\ln(1+L^{\tilde q}_T)\right]^{\frac{1}{1+2\lambda}}\right)\right]<+\infty.
\end{equation}
To show that $(\tilde q_t)_{t\in\T}\in \qcal(\xi,f)$, we have to prove that
\begin{equation}\label{eq:4.77}
\E_{\Q^{\tilde q}}\left[|\xi|+\int_0^T f(s,\tilde q_s)\,\dif s\right]<+\infty.
\end{equation}
Indeed, by \eqref{eq:4.3} with $\mu=\bar\mu$ and $\delta=\frac{1}{1+2\lambda}$, \eqref{eq:4.63} and \eqref{eq:4.76},  we have
\begin{equation}\label{eq:4.78}
\begin{array}{l}
\Dis \E_{\Q^{\tilde q}\left[|\xi|\right]}=\Dis \E\left[|\xi|L^{\tilde q}_T\right]
\leq \Dis \E\left[L^{\tilde q}_T\exp\left({\bar\mu^{-\frac{1}{1+2\lambda}}}\left[\ln(1+L^{\tilde q}_T)\right]^{\frac{1}{1+2\lambda}}\right)
\right]\vspace{0.2cm}\\
\hspace{3.6cm}\Dis +\E\left[|\xi|\exp\left(\bar\mu\left[\ln (1+|\xi|)\right]^{1+2\lambda}\right)\right]<+\infty.
\end{array}
\end{equation}
By an identical computation to \eqref{eq:4.45}, we deduce that for each $n\geq 1$,
\begin{equation}\label{eq:4.79}
Y_t=\E_{\Q^{\tilde q}}\left[\left. Y_{\tau^t_n}+\int_t^{\tau^t_n} f(s,\tilde q_s)\dif s\right|\F_t\right], \ \ t\in \T.
\end{equation}
Then, in light of \eqref{eq:4.79} with $t=0$, \eqref{eq:4.3} with $\mu=\tilde\mu$ and $\delta=\frac{1}{1+2\lambda}$, \eqref{eq:4.68} and \eqref{eq:4.76}, we have
$$
\sup_{n\geq 1}\E_{\Q^{\tilde q}}\left[\int_0^{\tau^0_n} f(s,\tilde q_s)\dif s\right]\leq |Y_0|+\E_{\Q^{\tilde q}}\left[\sup_{t\in\T}|Y_t|\right]<+\infty,
$$
which together with Fatou's lemma and \eqref{eq:4.78} yields \eqref{eq:4.77}. Sending $n$ to infinity in \eqref{eq:4.79} and applying Lebesgue's dominated convergence theorem, we get \eqref{eq:4.71}. The rest proof runs as (i) of \cref{thm:3.1}.

\subsection{Proof of Assertion (iv) of \cref{thm:3.1}\vspace{0.2cm}}

Assume that the core function $f$ satisfies (A0) and (A4), and that
\begin{equation}\label{eq:4.80}
|\xi|\in \bigcap_{\bar\mu>0}L\exp(\bar\mu (\ln L)^{\frac{1}{2}}).
\end{equation}
Consider a constant $\mu>\mu_0:=\gamma \sqrt{2T}$. First of all, in light of (A4), we have from \eqref{eq:3.2} that
\begin{equation}\label{eq:4.81}
U_t(\xi):=\essinf\limits_{q\in \bar\qcal(\xi,f)}\E_{\Q^q}\left[\left. \xi+\int_t^T f(s,q_s){\rm d}s\right|\F_t\right],\ \ t\in \T,\vspace{-0.1cm}
\end{equation}
where
$$
\begin{array}{l}
\bar\qcal(\xi,f):=\bigg\{
 (\F_t)\text{-progressively measurable~} \R^d \text{-valued process} \, (q_t)_{t\in\T}:\\
\Dis \hspace{2.2cm}\as,\ |q_t|\leq \gamma\ \ {\rm and}\ \ \E_{\Q^q}\left[|\xi|+\int_0^T f(s,q_s)\, \dif s\right]<+\infty\vspace{0.1cm}\\
\Dis \hspace{2.2cm}\text{with}\ {\dif \Q^q}:=\exp\Big(\int_{0}^{T}q_s\cdot \dif B_s-\frac{1}{2}\int_{0}^{T}| q_s|^2\dif s\Big)\, {\dif \p} \bigg\}.\vspace{0.2cm}
\end{array}
$$
Moreover, we see  from \cite[Lemma 2.6]{HuTang2018ECP} that for any $(\F_t)$-progressively measurable $\R^d$-valued process $(q_t)_{t\in\T}$ such that $\as$, $|q_t|\leq \gamma$, and for each $\bar\mu>\mu_0$, we have
\begin{equation}\label{eq:4.82}
\E\left[\left. \exp\left(\frac{1}{\bar\mu^2}\left|\int_t^T q_s\cdot \dif B_s\right|^2 \right) \right|\F_t\right]\leq \frac{1}{\sqrt{1-\frac{2\gamma^2}{\bar\mu^2}(T-t)}}<+\infty,\ \ t\in\T.
\end{equation}
Next, we show that $(\bar q_t)_{t\in\T}\in \bar\qcal(\xi,f)$, and then the space $\bar\qcal(\xi,f)$ is nonempty. Indeed, it follows from (A0), (A4) and \cref{rmk:3.0} that $\as$, $f(t,\bar q_t)=0$ and $|\bar q_t|\leq \gamma$, and then according to \eqref{eq:4.7} with some constant $q:=p\in (1,\frac{\mu}{\mu_0})$ and \eqref{eq:4.82} with $\bar\mu=\frac{\mu}{\sqrt{p}}>\mu_0$ and $t=0$ together with \eqref{eq:4.80}, we can deduce that there exists a constant $C_{\mu,p}>0$ depending only on $(\mu,p)$ such that
\begin{equation}\label{eq:4.83}
\begin{array}{l}
\Dis \E_{\Q^{\bar q}}\left[|\xi|+\int_0^T f(s,\bar q_s)\, \dif s\right]\leq \Dis 2\E\left[|\xi|L_T^{\bar q}\right]\leq 2\E\left[|\xi|\exp\left(\int_{0}^{T}\bar q_s\cdot \dif B_s\right)\right]\vspace{0.1cm}\\
\ \ \ \ \leq \Dis 2C_{\mu,p}\E\left[\exp\left(\frac{p}{\mu^2}\left|\int_0^T \bar q_s\cdot \dif B_s\right|^2 \right)\right]+2\E\left[|\xi|\exp\left(\mu \sqrt{\ln (1+|\xi|)}\right)\right] <+\infty.
\end{array}
\end{equation}
Hence, $(\bar q_t)_{t\in\T}\in \bar\qcal(\xi,f)$.\vspace{0.2cm}

Furthermore, by \eqref{eq:3.4} and (A4) we can deduce that $\as$, for each $z\in \R^d$,
$$
g(\omega,t,z)=\sup_{q\in \R^d} \left\{z\cdot q-f(\omega,t,q)\right\}\leq \sup_{q\in \R^d , |q|\leq \gamma} \left\{z\cdot q\right\}=\gamma |z|,
$$
which together with \cref{rmk:3.0} yields that the generator $g$ defined in \eqref{eq:3.4} satisfies assumptions (H0) and (H4) with $h_\cdot=0$. It then follows from (iv) of \cref{pro:2.2} and \eqref{eq:4.80} that BSDE \eqref{eq:3.3} admits a unique adapted solution $(Y_t,Z_t)_{t\in\T}$ such that
\begin{equation}\label{eq:4.85}
{\rm for\ each}\ \bar\mu>0,\ {\rm the\ process}\ (|Y_t|\exp(\bar\mu\sqrt{\ln(1+|Y_t|)}))_{t\in\T}\ {\rm belongs\ to\ class\ (D)}.
\end{equation}

To show the dual representation, we  need to further verify that $U_t(\xi)=Y_t$ for each $t\in \T$, where $U_t(\xi)$ is defined in \eqref{eq:4.81}. We first prove that for each $(q_t)_{t\in\T}\in \bar\qcal(\xi,f)$, it holds that
\begin{equation}\label{eq:4.86}
\E_{\Q^q}\left[\left.\xi+\int_t^Tf(s,q_s)\dif s\right|\F_t\right]\geq Y_t,\ \ t\in\T.
\end{equation}
In light of (A0), \eqref{eq:3.3} and \eqref{eq:3.4}, a same computation as that from \eqref{eq:4.33}-\eqref{eq:4.35} yields that for each $n\geq 1$,
\begin{equation}\label{eq:4.87}
\begin{array}{lll}
Y_t&\leq & \Dis \E_{\Q^q}\left[\left.Y_{\tau_n^t}+\int_t^{\tau_n^t}f(s,q_s)\dif s\right|\F_t\right]\vspace{0.2cm}\\
&=& \Dis \E\left[\left.Y_{\tau_n^t} \frac{L^q_{\tau_n^t}}{L^q_t}
\right|\F_t\right]+\E_{\Q^q}\left[\left.\int_t^{\tau_n^t}f(s,q_s)\dif s\right|\F_t\right], \ \ t\in \T,
\end{array}
\end{equation}
where the stopping time $\tau_n^t$ is defined in \eqref{eq:4.34}. According to \eqref{eq:4.7} with some constant $q:=p\in (1, \frac{\mu}{\mu_0})$ and \eqref{eq:4.82} with $\bar\mu=\frac{\mu}{p}>\mu_0$, we can deduce that there exists a constant $C_{\mu,p}>0$ depending only on $(\mu,p)$ such that for each $n\geq 1$ and $t\in\T$,
$$
\begin{array}{lll}
\Dis Y_{\tau_n^t} \frac{L^q_{\tau_n^t}}{L^q_t}&\leq & \Dis |Y_{\tau_n^t}| \exp\left(\int_t^{\tau_n^t} q_s\cdot \dif B_s\right)\vspace{0.1cm}\\
&\leq & \Dis C_{\mu,p}\exp\left(\frac{p}{\mu^2}\left|\int_t^{\tau_n^t} q_s\cdot \dif B_s\right|^2 \right)+
|Y_{\tau_n^t}|\exp\left\{\mu \sqrt{\ln\left(1+|Y_{\tau_n^t}|\right)}\right\}
\end{array}
$$
and
$$
\begin{array}{lll}
\Dis \E\left[\left|\exp\left(\frac{p}{\mu^2}\left|\int_t^{\tau_n^t} q_s\cdot \dif B_s\right|^2\right)\right|^p \right]&=& \Dis \E\left[\exp\left(\frac{p^2}{\mu^2}\left|\int_t^{\tau_n^t} q_s\cdot \dif B_s\right|^2 \right)\right]\vspace{0.1cm}\\
&\leq& \Dis \frac{1}{\sqrt{1-\frac{2p^2\gamma^2}{\mu^2}(T-t)}}<+\infty,
\end{array}
$$
which together with \eqref{eq:4.85} yield that for each $t\in\T$, the sequence of random variables $Y_{\tau_n^t} \frac{L^q_{\tau_n^t}}{L^q_t}$ are uniformly integrable. On the other hand, since $(q_t)_{t\in\T}\in \bar\qcal(\xi,f)$, we have
$$
\E_{\Q^q}\left[\int_0^T f(s,q_s)\dif s\right]<+\infty.
$$
Thus, the desired assertion \eqref{eq:4.86} follows immediately by sending $n$ to infinity in \eqref{eq:4.87}.\vspace{0.2cm}

Next, we set $\tilde q_s\in \partial g(s,Z_s)$ for each $s\in \T$ and prove that $(\tilde q_t)_{t\in\T}\in \bar\qcal(\xi,f)$ and
\begin{equation}\label{eq:4.88}
Y_t=\E_{\Q^{\tilde q}}\left[\left.\xi+\int_t^T f(s,\tilde q_s)\dif s\right|\F_t\right],\ \ t\in\T.
\end{equation}
Since $f(s,\tilde q_s)=Z_s\cdot \tilde q_s-g(s,Z_s)$, by (A4) we deduce that $\as$, $|\tilde q_t|\leq \gamma$, and by \eqref{eq:4.82} we obtain that for some constant $p\in (1,\frac{\mu}{\mu_0})$,
\begin{equation}\label{eq:4.89}
\Dis\E\left[\exp\left(\frac{p}{\mu^2}\left|\int_0^T  \tilde q_s\cdot \dif B_s\right|^2 \right)\right]\leq \E\left[\exp\left(\frac{p^2}{\mu^2}\left|\int_0^T  \tilde q_s\cdot \dif B_s\right|^2 \right)\right]\leq \frac{1}{\sqrt{1-\frac{2p^2\gamma^2}{\mu^2}(T-t)}}<+\infty.
\end{equation}
To show that $(\tilde q_t)_{t\in\T}\in \bar\qcal(\xi,f)$, we have to prove that
\begin{equation}\label{eq:4.90}
\E_{\Q^{\tilde q}}\left[|\xi|+\int_0^T f(s,\tilde q_s)\, \dif s\right]<+\infty.
\end{equation}
Indeed, by \eqref{eq:4.7} with $q:=p$, \eqref{eq:4.89} and \eqref{eq:4.80} we deduce that there exists a constant $\bar C_{\mu,p}>0$ depending only on $(\mu,p)$ such that
\begin{equation}\label{eq:4.91}
\begin{array}{l}
\Dis \E_{\Q^{\tilde q}}\left[|\xi|\right]\leq \Dis \E\left[|\xi|L^{\tilde q}_T\right]\leq \E\left[|\xi|\exp\left(\int_0^T \tilde q_s\cdot \dif B_s\right)\right]\vspace{0.1cm}\\
\hspace{1.35cm}\leq \Dis \bar C_{\mu,p}\E\left[\exp\left(\frac{p}{\mu^2}\left|\int_0^T \tilde q_s\cdot \dif B_s\right|^2 \right)\right]+\E\left[|\xi|\exp\left(\mu \sqrt{\ln (1+|\xi|)}\right)\right]<+\infty.
\end{array}
\end{equation}
By an identical computation to \eqref{eq:4.45}, we obtain that for each $n\geq 1$,
\begin{equation}\label{eq:4.92}
\begin{array}{lll}
Y_t&=& \Dis \E_{\Q^{\tilde q}}\left[\left.Y_{\tau_n^t}+\int_t^{\tau_n^t}f(s,{\tilde q}_s)\dif s\right|\F_t\right]\vspace{0.1cm}\\
&=& \Dis \E\left[\left.Y_{\tau_n^t} \frac{L^{\tilde q}_{\tau_n^t}}{L^{\tilde q}_t}
\right|\F_t\right]+\E_{\Q^{\tilde q}}\left[\left.\int_t^{\tau_n^t}f(s,{\tilde q}_s)\dif s\right|\F_t\right], \ \ t\in \T.
\end{array}
\end{equation}
On the other hand, by applying \eqref{eq:4.7}, \eqref{eq:4.85} and \eqref{eq:4.89} we deduce that the sequence of random variables $|Y_{\tau_n^t}| L^{\tilde q}_{\tau_n^t}/L^{\tilde q}_t$ is uniformly integrable for each $t\in\T$. Then, in light of \eqref{eq:4.92} with $t=0$, we have
$$
\sup_{n\geq 1}\E_{\Q^{\tilde q}}\left[\int_0^{\tau^0_n} f(s,\tilde q_s)\, \dif s\right]\leq |Y_0|+\sup_{n\geq 1}\E\left[|Y_{\tau_n^0}| L^{\tilde q}_{\tau_n^0}\right]<+\infty,
$$
which together with Fatou's lemma and \eqref{eq:4.91} yields \eqref{eq:4.90}. By letting $n\To \infty$ in \eqref{eq:4.92}, we get \eqref{eq:4.88}. The rest proof runs as (i) of \cref{thm:3.1}.




\setlength{\bibsep}{4pt}


\end{document}